\documentclass[12pt]{article}
\usepackage[cm]{fullpage}
\usepackage[small]{titlesec}
\usepackage[cm]{fullpage}

\usepackage{amsmath,amscd, float,rotating}
\usepackage{pb-diagram}

\usepackage{latexsym}
\usepackage{amsfonts,color}
\usepackage{amsmath}
\usepackage{amssymb}
\usepackage{txfonts}

\numberwithin{equation}{section}
\newcommand{\qed}{\hfill \ensuremath{\Box}}

\def\XXint#1#2#3{{\setbox0=\hbox{$#1{#2#3}{\int}$}
\vcenter{\hbox{$#2#3$}}\kern-.5\wd0}}

\newcommand{\beq}{\begin{equation}}
\newcommand{\eeq}{\end{equation}}

\newcommand{\beqs}{\begin{eqnarray*}}
\newcommand{\eeqs}{\end{eqnarray*}}

\newcommand{\CC}{\mathbb{C}}

\newcommand{\N}{\mathbb{N}}
\newcommand{\Z}{\mathbb{Z}}
\newcommand{\cO}{\mathcal{O}}

\newcommand{\ddbar}{\sqrt{-1}\partial\bar{\partial}}
\newcommand{\barX}{\overline{X}}

\begin{document}
\newcounter{remark}
\newcounter{theor}
\setcounter{remark}{0} \setcounter{theor}{1}
\newtheorem{claim}{Claim}
\newtheorem{theorem}{Theorem}[section]
\newtheorem{proposition}{Proposition}[section]
\newtheorem{lemma}{Lemma}[section]
\newtheorem{definition}{Definition}[section]
\newtheorem{conjecture}{Conjecture}[section]
\newtheorem{corollary}{Corollary}[section]
\newenvironment{proof}[1][Proof]{\begin{trivlist}
\item[\hskip \labelsep {\bfseries #1}]}{\end{trivlist}}
\newenvironment{remark}[1][Remark]{\addtocounter{remark}{1} \begin{trivlist}
\item[\hskip \labelsep {\bfseries #1
\thesection.\theremark}]}{\end{trivlist}}
\newenvironment{example}[1][Example]{\addtocounter{remark}{1} \begin{trivlist}
\item[\hskip \labelsep {\bfseries #1
\thesection.\theremark}]}{\end{trivlist}}
~

\begin{center}
{\large \bf
Asymptotic expansions of complete K\"ahler-Einstein metrics with finite volume on quasi-projective manifolds}
\bigskip\bigskip

{Xumin Jiang$^{*}$ and Yalong Shi$^{**}$} \\

\bigskip

\end{center}

\begin{abstract}

{\footnotesize We give an elementary proof to the asymptotic expansion formula of Rochon-Zhang for the unique complete K\"ahler-Einstein metric of Cheng-Yau, Kobayashi, Tian-Yau and Bando on quasi-projective manifolds. The main tools are the solution formula for second order ODE's with constant coefficients and spectral theory for Laplacian operator on a closed manifold.}

\end{abstract}

\tableofcontents

\section{Introduction}

Complete non-compact K\"ahler-Einstein metrics play an important role in several complex variables and geometry as observed by C. Fefferman \cite{Fef} , S. Cheng and S.T.Yau \cite{CY2} in the 1970's. The existence of such metrics in strictly pseudoconvex domains with smooth boundary in $\CC^n$ is proved by Cheng-Yau \cite{CY1} extending Yau's solution of Calabi's conjecture \cite{Yau}. In \cite{CY1}, boundary regularity for the solution is also discussed. Later, a more precise boundary regularity theorem and an asymptotic expansion of the solution near boundary are obtained by  J. Lee and R. Melrose in 1982 \cite{LM}. Later, the coefficients of Lee-Melrose's expansion have been calculated by J. Lee \cite{Le} and R. Graham \cite{Gra}. See also the recent work of Q. Han and X. Jiang \cite{HJ2} for another proof for the asymptotic expansion formula.

If the manifold is not an Euclidean domain, up to now, all the known examples of complete K\"ahler-Einstein metrics with negative Einstein constants are quasi-projective. Let $\barX$ be a smooth projective manifold of complex dimension $n$, $D\subset \barX$ a smooth hypersurface such that $K_{\barX}+D$ is ample. In the 1990's, in a series of works Cheng-Yau \cite{CY2}, R. Kobayashi \cite{Ko}, Tian-Yau \cite{TY},  S. Bando \cite{Ban} proved that the quasi-projective manifold $X:=\barX\setminus D$ admits a unique complete K\"ahler-Einstein metric $\omega_{KE}$ with finite volume and $Ric(\omega_{KE})=-\omega_{KE}$. In fact, their results also allows $D$ to be a simple normal crossing divisor, and $K_{\barX}+D$ is only big and nef, and ``ample modulo $D$ ".  For the K\"ahler-Ricci flow approach to the existence of such metrics, please look at the work of J. Lott and Z. Zhang \cite{LZ}.

The asymptotic expansion of these quasi-projective K\"ahler-Einstein metrics is first studied by G. Schumacher in 1998.  By adjunction formula $K_D=(K_{\barX}+D)|_D>0$, so Yau's theorem guarantees the existence of a unique K\"ahler-Einstein metric $\omega_D$ satisfying $Ric(\omega_D)=-\omega_D$. Schumacher proved in \cite{Sch} that the restriction of $\omega_{KE}$ to directions parallel to $D$ will converges to $\omega_D$.  Later, a systematical study is done by D. Wu in his thesis \cite{Wu} in 2006 by analyzing the mapping property of the linearized complex Monge-Amp\`ere operator on weighted Cheng-Yau H\"older rings. Wu obtained an asymptotic expansion of the solution $u$ to the complex Monge-Amp\`re equation in terms of powers of $\sigma=(\log\|s\|^2)^{-1}$, where $s$ is the defining section of $D$. However, as observed by  F. Rochon and Z. Zhang in 2012 \cite{RoZh}, $\sigma\log\sigma$-term should appear in general, depending on the normal bundle of $D$. In Rochon-Zhang \cite{RoZh}, a more precise asymptotic expansion is obtained using the so called ``b-calculus", developed by Melrose and his students.

For asymptotic expansions of other types of canonical metrics, for example complete Calabi-Yau metrics or conic K\"ahler-Einstein metrics, we refer the reader to the works of B. Santoro \cite{San}, T. Jeffres, R. Mazzeo and Y. Rubinstein \cite{JMR}, H. Yin and K. Zheng \cite{YZ}.

In this paper, we will give another proof of Rochon-Zhang's theorem by elementary tools, namely, besides rescaled interior Schauder estimates, the key tools are spectral decomposition for Laplacian operators on closed manifolds and the elementary theory of second order ordinary differential equations with constant coefficients. See also L. Andersson, P. Chru\'sciel and H. Friedrich \cite{ACF1982CMP}, H. Jian and X. Wang \cite{JW}, and Han-Jiang \cite{HJ1} \cite{HJ2} for the ODE iteration method. Even though our result is not new, this elementary approach is interesting in itself, and the authors expect it
to be useful in other geometric problems.

The main theorem of this paper is:
\begin{theorem}\label{main}
Let $\omega_{KE}=\omega+\ddbar u$ be the unique complete K\"ahler-Einstein metric with finite volume on $X=\barX\setminus D$, and let $x=(-\log r^2)^{-1}$, where $r$ is the distance to $D$ with respect to some fixed K\"ahler metric on $\barX$. Then we have a poly-homogeneous asymptotic expansion for $u$:
$$u\sim \sum_{i\in I}\sum_{j=0}^{N_i} c_{i,j} x^i(\log x)^j,$$
where $I$ is the index set determined by the eigenvalues of the Laplacian operator of the unique K\"ahler-Einstein metric on $D$ and $c_{i,j}$'s are smooth functions on $D$, regarded as functions in a neighborhood of $D$ via the Tubular Neighborhood Theorem. The precise meaning of the above expansion is that
$$u-\sum_{i=0}^k \sum_{j=0}^{N_i} c_{i,j} x^i(\log x)^j=\bold{O}(x^{k_+}),$$
where $k_+$ is the next term of $k$ in $I$.
\end{theorem}

In this paper, $\bold{O}(x^N)$, for any real number $N>0$, denotes a function $\psi$ such that, for integers $k, l\geq 0$,
$$|(x\partial_x)^l\partial^k_{z',\bar z'}\psi|\leq C_{k,l,N} x^{N}.$$

\begin{remark}
It is obvious that in the statement of Theorem \ref{main}, we can replace $x$ by $(-\log\|s\|)^{-1}$, where $s$ is the defining section of $D$.
\end{remark}

Usually, the coefficients of an asymptotic expansion formula in a geometric problem will carry important geometric informations. For example the famous heat kernel expansion and Bergman kernel expansion play very important role in Riemannian geometry and K\"ahler geometry. Let's also mention the boundary asymptotic expansion of conformally compact Einstein metrics, which is very useful in conformal geometry. It is expected that the coefficients of the asymptotic expansion in Theorem \ref{main} will also carry interesting geometric informations. We leave this problem to a future work.\footnote{Some initial terms have already been carried out by Rochon-Zhang \cite{RoZh}. See also Lemma \ref{lem-c11}.} 

Recently, J. Sun and S. Sun studied the log K-stability of polarized Riemann surface with standard cusp singularities \cite{SS}. An important ingredient of their proof is a precise estimate of the Bergman kernel near the cusp singularity and in the neck region, which in turn requires better asymptotic behaviors of the hyperbolic metric near the singularity. The result and method of this paper should be helpful to attack the higher dimensional log K-stability problem.

The paper is organized as follows: In \S \ref{review}, we recall the basic facts concerning the construction of finite volume complete quasi-projective K\"ahler-Einstein metrics, including Cheng-Yau's quasi-coordinate map and their H\"older spaces. We shall derive some basic properties that will be used in later sections, and obtain the leading term of the solution via Cheng-Yau's maximum principle on complete non-compact manifolds. Then in \S \ref{holo coord}, we compute the linearization of the associated complex Monge-Amp\`ere equation in local holomorphic charts. Since we need to work ``semi-globally", we shall need another set of coordinates that is not holomorphic in general, namely coordinates from the Tubular Neighborhood Theorem. Since the holomorphic version of the Tubular Neighborhood Theorem does not hold in general, this non-holomorphic coordinate transformation causes most of the complication of this paper. The detailed computation is included in the appendix for the convenience of the readers.  Then we show in \S \ref{formal} that one can derive a series of formal approximate solutions. They can be viewed as a formal asymptotic expansion. The $x\log x$-term and index set appear naturally in this process. Even though this part is not logically required in our proof, we feel that it may be helpful for the readers to understand our proof. Finally, in \S \ref{iterate}, we use the solution formula of second order ODE to do induction, and hence finish the proof.\\

\noindent {\bf{Acknowledgements:}} This work is carried out during the second author's visit at Rutgers University by the support of Hwa Ying Foundation. He would like to  thank the Foundation for its support, thank Professor Jian Song for his invitation and the Department of Mathematics at Rutgers University for its hospitality. Both authors thank Professor Jian Song for his interest in this work.

 \section{Generalities on the complete K\"ahler-Einstein metrics}\label{review}

Let $D$ be a smooth hypersurface in $\barX$. As is well known, $D$ determines a unique holomorphic line bundle $\cO(D)$, and $D=(s)$ is the divisor of a (unique up to a non-zero constant factor) holomorphic section $s\in H^0(X,\cO(D))$. In the following, we always assume that $L:=K_X+D$ is ample. We choose a smooth Hermitian metric $h$ on $L$, such that the curvature form $\sqrt{-1}\Theta_L>0$. We also choose a smooth metric on $\cO(D)$, locally of the form $e^{-\varphi}$. Locally at some point $p\in D$, we choose coordinates s.t. $D$ is defined by $\{z_n=0\}$, then $||s||^2=|z_n|^2 e^{-\varphi}$.

Now consider the following Carlson-Griffiths \cite{CG} reference metric on $X=\barX\setminus D$
$$\omega:=\sqrt{-1}\Theta_L-\ddbar\log\Big(\log\frac{1}{\epsilon||s||^2}\Big)^2.$$
Direct computation shows that when $\epsilon<<1$, it is indeed a complete K\"ahler metric with finite volume. For simplicity, we rescale $s$ by $\sqrt{\epsilon}$, and {\bf from now on, we always assume $\epsilon=1$}. As observed by Kobayashi \cite{Ko} and Tian-Yau \cite{TY}, $(X, \omega)$ has bounded geometry of infinite order, which means that one can find a family of holomorphic maps of maximal rank from balls of definite size in $\CC^n$ into $X$ (the so called ``quasi-coordinates"), whose images cover $X$, such that the pull backs of $\omega$ to the pre-images  are uniformly equivalent to the standard Euclidean metric and all the derivatives of  the pull-back metric tensor are uniformly bounded. If we choose local holomorphic coordinates $(z_1,\dots, z_n)$ such that $D$ is defined by $z_n=0$, then typical quasi-coordinates $(z_1,\dots, z_{n-1}, w)$ near $D$ can be defined by
$$z_n=exp\Big(\frac{1+\eta}{1-\eta}\cdot\frac{w+1}{w-1}\Big), $$
where $0<\eta<1$ and $|w|\leq \frac{2}{3}$.
According to Cheng-Yau \cite{CY1}, one can define the global H\"older norm $\|u\|_{k,\alpha}$ to be the supremum of the Euclidean $C^{k,\alpha}$ norms of the pull-back of $u$ on quasi-coordinate charts. We define  $C^{k,\alpha}(X)$ to be the space of $C^k$ functions $u$ such that $\|u\|_{k,\alpha}<\infty$. Using Cheng-Yau's method, Kobayashi \cite{Ko}, Tian-Yau \cite{TY} and Bando \cite{Ban} proved the following existence theorem:

\begin{theorem}[Kobayashi, Tian-Yau, Bando]\label{existence}
There exists a unique complete K\"ahler-Einstein metric $\omega_{KE}=\omega+\ddbar u$ on $X$ satisfying $$Ric(\omega_{KE})=-\omega_{KE}.$$ 
Moreover, $u\in C^{k,\alpha}(X)$ for any $k\in\N$ and $0<\alpha<1$ and $\omega_{KE}$ is equivalent to $\omega$.
\end{theorem}

The uniqueness follows from Yau's Schwarz lemma. To prove the existence, one solves the complex Monge-Amp\`ere equation
\beq\label{eqn:main}
\log\frac{(\omega+\ddbar u)^n}{\omega^n}-u=f,
\eeq
where $f$ is a smooth function on $X$ such that $Ric(\omega)+\omega=\ddbar f$. As in \cite{Ko} and \cite{TY}, $f\in C^{k,\alpha}(X)$ for any $k\in\N$ and $0<\alpha<1$. In fact, if we write the bundle metric $h$ on $L$ locally as $e^{-\varphi}/\phi$ where  $\Phi=(\sqrt{-1})^n\phi dz_1\wedge d\bar z_1\dots dz_n\wedge d\bar z_n$ is a smooth volume form on $\barX$,  then $f$ can be chosen as $\log\frac{\Phi}{\Psi}$, where and $\Psi=\|s\|^2(-\log\|s\|^2)^2\omega^n$. Since $D$ is smooth and locally $s=z_n$, a direct computation shows that $\Psi$ extends to a continuous volume form on $\barX$, hence $f$ extends to a continuous function on $\barX$. In fact $f|_D$ is a smooth function on $D$.

\begin{lemma}\label{lem-f}
If we choose the bundle metric on $L$ such that $\sqrt{-1}\Theta_L |_D$ is the canonical K\"ahler-Einstein metric $\omega_D$ satisfying $Ric(\omega_D)=-\omega_D$,
\footnote{This is always possible by \cite{Sch}.}  and denote $\sigma:=-\log(\|s\|^2)$, then there is a constant $c_0$ such that $f=-c_0+O(\sigma^{-1})$ in a neighborhood of $D$.  
\end{lemma}

\begin{proof} We can compute in local coordinates. First, it is easy to see that $f-f|_D=O(\sigma^{-1})$. So it suffices to show $f|_D\equiv -c_0$ for some constant $c_0$.

 By direct computation, we have
\beqs
\Psi(z',0) &=& 2ne^{-\varphi(z',0)}\big(\sqrt{-1}\Theta_L|_D\big)^{n-1}\wedge \big(\sqrt{-1}dz_n\wedge d\bar z_n\big)\\
&=& 2ne^{-\varphi(z',0)}\omega_D^{n-1}\wedge \big(\sqrt{-1}dz_n\wedge d\bar z_n\big),
\eeqs
and
$$\Phi(z',0)=\phi(z',0)(\sqrt{-1})^n dz_1\wedge d\bar z_1\wedge\dots\wedge dz_n\wedge d\bar z_n.$$
So we have
\beqs
f |_D (z',0) &=& \log \phi(z',0)+\varphi(z',0)-\log\det(g^D_{\alpha\bar\beta})(z',0) +c_n.
\eeqs
So we have
\beqs
\sqrt{-1}\partial_D\bar\partial_D f|_D &=& (\sqrt{-1}\Theta_L)|_D+Ric(\omega_D) \\
&=&\omega_D +Ric(\omega_D)=0.
\eeqs
So we have $f|_D\equiv - c_0$ for some constant $c_0$. \qed
\end{proof}

Using this Lemma, we can find the leading order behavior of $u$ near $D$:

\begin{lemma}\label{lem-u-c0}
For the same constant $c_0$ as above, we have $u=c_0+O(\sigma^{-1}\log\sigma)$.
\end{lemma}

\begin{proof}
The main tool of our proof is the following version of Cheng-Yau's maximum principle in \cite{TY}:
\begin{lemma}[Cheng-Yau's Maximum Principle]\label{lem-CY}
Let $(M^n,g)$ be a complete Riemannian manifold with sectional curvature bounded from below. Let $\varphi$ be a smooth function on $M$ such that $\sup_M \varphi<\infty$. Then there exist a sequence of points $\{p_i\}\subset M$ such that 
$$\lim_i \varphi(p_i)=\sup_M\varphi,\quad \lim_i |\nabla\varphi|_g(p_i)=0,\quad \lim\sup_i Hess\ \varphi (p_i)\leq 0.$$
\end{lemma}

 Take a sufficiently small neighborhood $U$ of $D$ such that $\sigma^{-1}\log\sigma$ is strictly positive on $\partial U$. We shall find large positive constants $a,b$ such that 
$$-a \sigma^{-1}\log\sigma \leq u-c_0 \leq b\sigma^{-1}\log\sigma$$
in $U$. For this, we use the test functions $M_+:=u-c_0+a \sigma^{-1}\log\sigma$  and $M_-:=u-c_0-b \sigma^{-1}\log\sigma$, with constants $a, b>0$ to be determined later. We take $M_-$ for example, and the discussion for $M_+$ is the same.

First we assume $M_-<0$ on $\partial U$. This is true if $b$ is large enough, since $u$ is bounded. If $\sup_{U\setminus D} M_-\geq 0$, by Cheng-Yau's maximum principle, we can find a sequence $p_i\in U\setminus D$ such that $M_-(p_i)\to \sup_{U\setminus D} M_-$ and 
$$\lim\sup_i Hess M_-(p_i)\leq 0.$$
So we have
\beqs
\sup M_-&=&\lim (f+u)(p_i)+\lim(-c_0-f-b\sigma^{-1}\log \sigma)(p_i)\\
&=& \lim \log \frac{(\omega+\ddbar u)^n}{\omega^n}(p_i)+\lim(-c_0-f-b\sigma^{-1}\log \sigma)(p_i)\\
&\leq& \lim \log \frac{(\omega+b\ddbar \sigma^{-1}\log\sigma)^n}{\omega^n}(p_i)+\lim(-c_0-f-b\sigma^{-1}\log \sigma)(p_i)\\
&\leq & \lim \Big(-c_0-f(p_i)-\big(\frac{3}{2}b\sigma^{-1}+o(\sigma^{-1})\big)(p_i)\Big).
\eeqs
Since $f+c_0=O(\sigma^{-1})$, for $b$ large enough, we must have $M_-\leq 0$.  \qed
\end{proof}

For the higher derivatives of $u$, we have:

\begin{lemma}\label{lem-v_higher}
Under the same assumption as Lemma \ref{lem-f}, for a solution $u$ of \eqref{eqn:main}, we have
\beq \label{eq-nablaKu}
|\nabla^k u|_\omega \leq C_k \sigma^{-1} \log \sigma
\eeq
for any integer $k\geq 1$.
\end{lemma}

\begin{proof}
Denote $v= u-c_0$. By Lemma \ref{lem-u-c0}, $v=O(\sigma^{-1}\log \sigma)$. And $v$ satisfies the equation
\beqs
\log\frac{(\omega+\ddbar v)^n}{\omega^n}-v=f+c_0.
\eeqs
In quasi-coordinates, we can rewrite the equation as 
\beqs
A^{\bar j i} \partial_i\partial_{\bar j} v-v=f+c_0,
\eeqs
where 
$$A^{\bar j i}=\int_0^1 \Big((g_{k\bar l}+t u_{k\bar l})^{-1}\Big)^{\bar j i} dt.$$
By Theorem \ref{existence}, we can view this as a uniformly elliptic linear equation on $v$ with smooth coefficients. Since by Lemma \ref{lem-u-c0}, $v=O(\sigma^{-1}\log \sigma)$, the lemma follows from classical interior Schauder estimates if we have the following:
\begin{claim}
For any integer $k\geq 0$, we have
$$|\nabla^k(f+c_0)|_\omega=O(\sigma^{-1}).$$
\end{claim}
We shall prove this by mathematical induction. The $k=0$ case is proved in Lemma \ref{lem-f}. Now we assume $|\nabla^i(f+c_0)|_\omega=O(\sigma^{-1})$ for $i=0,\dots k-1$. In quasi-coordinates $(z_1,\dots, z_{n-1}, w)=:(z',w)$, we have $\partial^i_{z',\bar z'}\partial^j_{w,\bar w} (f+c_0)=O(\sigma^{-1})$ for all multi-index $i$ and integer $j$ such that  $|i|+j\leq k-1$. Now for any $|i|+j=k-1$, we have
$$\partial_{z', \bar z'}^i\partial_{w,\bar w}^j(f+c_0)=(\log|z_n|^2)^{-1} a(z',w).$$
Since $f$ is in fact smooth with respect to $\hat x:=(-\log |z_n|^2)^{-1}$, we conclude that any derivatives of $a(z', w)$ with respect to the coordinates $(z',w)$ are still bounded.
If we take another $z'$ or $\bar z'$ derivative, then obviously we still have $O(\sigma^{-1})$. On the other hand, it is direct to check that
$$\frac{\partial}{\partial w}(-\log |z_n|^2)^{-1}=O((-\log |z_n|^2)^{-1}).$$
So we have  $\partial^i_{z',\bar z'}\partial^j_{w,\bar w} (f+c_0)=O(\sigma^{-1})$ for all multi-index $i$ and integer $j$ such that  $|i|+j\leq k$.
 \qed
\end{proof}

Choose local holomorphic coordinates $(z_1,\dots,z_n)$ such that locally, $D=\{z_n=0\}$. Write $z_n:=re^{i\theta}$. At this moment, we define
$$x= (-\log r^2)^{-1}.$$
Another observation about $f$ and $u$, which is of crucial importance for our later discussions, is that they are essentially independent of $\theta$:
\begin{lemma}\label{lem:theta}
We have
$$\partial^k_\theta f=\bold{O}(x^\infty),  \partial^k_\theta u=\bold{O}(x^\infty)$$
for any $k\geq 1$. Here $\bold{O}(x^\infty)$ means a function $\psi$ satisfying
$$|\partial^k_{z',\bar z'}\partial^l_x\psi|\leq C_{k,l,N} x^N$$
for any $k\geq 0, l\geq 0$ and $N\in \N$.
\end{lemma}

\begin{proof} We only need to prove the $l=0$ case. Recall that all the derivatives of $f$ and $u$ with respect to quasi-local coordinates are uniformly bounded. If we choose local holomorphic coordinates $(z_1,\dots, z_n)$ such that $D$ is defined by $z_n=0$, then the quasi-coordinates $(z_1,\dots, z_{n-1}, w)$ can be defined by
$$z_n=exp\Big(\frac{1+\eta}{1-\eta}\cdot\frac{w+1}{w-1}\Big), $$
where $0<\eta<1$ and $|w|\leq \frac{2}{3}$. Then we have
\beqs
\frac{\partial}{\partial\theta}&=& i\big(z_n\frac{\partial}{\partial z_n}-\bar z_n \frac{\partial}{\partial\bar z_n}\big)\\
&=&-i\frac{1-\eta}{2(1+\eta)}\Big((1-w)^2\frac{\partial}{\partial w}-(1-\bar w)^2\frac{\partial}{\partial \bar w}\Big).
\eeqs
 On the other hand, we have
$$x^{-1}\sim \frac{1-|w|^2}{|1-w|^2}.$$
Direct computation shows that
$$x^2\frac{\partial}{\partial x}=Re\Big(z_n\frac{\partial}{\partial z_n}\Big)=\frac{\eta-1}{2(1+\eta)}Re\Big((1-w)^2\frac{\partial}{\partial w}\Big).$$
This implies that the coefficients of $x^{-1}\partial_\theta$ and $x\partial_x$ are bounded and smooth.
So if $\psi$ is in Cheng-Yau's H\"older space $C^{N,\alpha}(M,g_0)$, then $\partial^l_{z',\bar z'}\partial^k_\theta\psi=O(x^k)$ for $k+l\leq N$.  Since any such function $\partial^l_{z',\bar z'}\partial^k_\theta\psi$ must be periodic in $\theta$, we have $\int_0^{2\pi}\partial^l_{z',\bar z'}\partial^{k-1}_\theta\psi   =0 $ when $k-1\geq 1$. This means we can find (for fixed $z'$ and $x$) a $\theta_0$ such that $\partial^l_{z',\bar z'}\partial^{k-1}_\theta\psi(\theta_0)=0$. So we can integrate the $\theta$ variables from $\theta_0$ to get $\partial^l_{z',\bar z'}\partial^{k'}_\theta\psi=O(x^k)$ for any $1\leq k'\leq k$. 

Finally, since both $f$ and $u$ are in all $C^{k,\alpha}(X)$, we get the result. \qed
\end{proof}

\section{Linearized operator under local coordinates}\label{holo coord}

In this section, we shall compute the linearized complex Monge-Amp\`ere operator in local coordinate charts. This will we used in the next two sections to derive the asymptotic expansion.
Choose local holomorphic coordinates $(z_1,\dots,z_n)$ as before, such that locally $D=\{z_n=0\}$. Recall that $z_n:=re^{i\theta}$ and
$$x= (-\log r^2)^{-1}.$$
Also recall that
$$\sigma=-\log\|s\|^2=-\log r^2+\varphi=\frac{1}{x}+\varphi=\frac{1+x\varphi}{x}.$$
If locally $\Theta_L=\vartheta_{i\bar j}dz_i\wedge d\bar z_j$, and write $\omega:=\sqrt{-1}g_{i\bar j}dz_i\wedge d\bar z_j$, then we have 
\beqs
g_{i \bar{j}}&=&\vartheta_{i \bar{j}}-\partial_i \partial_{\bar{j}} \log \Big(-\log r^2+\varphi \Big)^2\\
&=& \vartheta_{i \bar{j}}-\frac{2x\varphi_{i \bar{j}}}{1+x\varphi }+N_i N_{\bar j},
\eeqs
where
\begin{align*}
N_i =\sqrt{2}\frac{(x^{-1})_i+\varphi_i}{x^{-1}+\varphi}= \frac{\sqrt{2}x}{r}\cdot  \frac{-\delta_{in}e^{-\sqrt{-1}\theta}+r\varphi_i}{1+x\varphi}.
\end{align*}

To compute $g^{\bar{j}i}$, set
\begin{align*}
\tilde{\vartheta}_{i\bar{j}}=\vartheta_{i\bar{j}}-\frac{2x\varphi_{i \bar{j}}}{1+x\varphi } 
\end{align*}
which is positive definite near $D$. Then a direct computation shows that

\begin{align*}
g^{\bar j l}&=\tilde{\vartheta}^{\bar j l}- \frac{\tilde{\vartheta}^{\bar j p}N_p \tilde{\vartheta}^{\bar q l}N_{\bar{q}}}{1+\tilde{\vartheta}^{\bar p q}N_{\bar{p}}N_q},\\
&=\tilde{\vartheta}^{\bar j l}- \frac{\tilde{\vartheta}^{\bar j p}\tilde N_p \tilde{\vartheta}^{\bar q l}\tilde N_{\bar{q}}}{\frac{x^2}{2x^2}+\tilde{\vartheta}^{\bar pq}\tilde{N}_{\bar{p}}\tilde{N}_q},
\end{align*}
so we derive
\begin{align}
\begin{split}\label{eq-GAB}
g^{\bar\beta \alpha} &= \tilde\vartheta^{\bar\beta\alpha}-\frac{\tilde\vartheta^{\bar n\alpha}\tilde\vartheta^{\bar\beta n}}{\tilde\vartheta^{\bar n n}}+O(r)=\Big( \vartheta^{\bar\beta\alpha}-\frac{\vartheta^{\bar n\alpha}\vartheta^{\bar\beta n}}{\vartheta^{\bar n n}}\Big)|_D+O(x),\\
g^{\bar nn} &=\frac{r^2}{2x^2}(1+2x\varphi)+O(r^2),\\
g^{\bar n \alpha} &= \Big(\tilde\vartheta^{\bar\beta\alpha}-\frac{\tilde\vartheta^{\bar n\alpha}\tilde\vartheta^{\bar\beta n}}{\tilde\vartheta^{\bar n n}}\Big) \varphi_{\bar\beta}\bar z_n +\frac{\tilde\vartheta^{\bar n\alpha}}{\tilde\vartheta^{\bar nn}}\frac{r^2}{2x^2}(1+2x\varphi)+O(r^2),\\
g^{\bar\beta n} &= \Big(\tilde\vartheta^{\bar\beta\alpha}-\frac{\tilde\vartheta^{\bar n\alpha}\tilde\vartheta^{\bar\beta n}}{\tilde\vartheta^{\bar n n}}\Big) \varphi_{\alpha} z_n +\frac{\tilde\vartheta^{\bar \beta n}}{\tilde\vartheta^{\bar nn}}\frac{r^2}{2x^2}(1+2x\varphi)+O(r^2).
\end{split}
\end{align}
Notice that the $(n-1)\times(n-1)$ matrix
$$\Big( \vartheta^{\bar\beta\alpha}-\frac{\vartheta^{\bar n\alpha}\vartheta^{\bar\beta n}}{\vartheta^{\bar n n}}\Big)|_D$$
on $D$ is exactly the inverse of $(\vartheta_{\alpha\bar\beta})|_D$. We write it as $\eta^{\bar\beta\alpha}$.

For an operator $T_0$, defined as
\begin{align}\label{eq-TForm}
T_0 v= \eta^{\bar{\beta}\alpha} v_{\alpha \bar{\beta}}+ x^2 Re( C_{\alpha \bar{n}}^*\frac{\partial v_{\alpha}}{\partial x})+\frac{1}{2}x^2 v_{xx}+ xv_x,
\end{align}
we say $T=T_0+ O(x)$, if
$T= T_1+ T_\infty$, where $T_1$ has same form \eqref{eq-TForm} as $T_0$, and the coefficients of $T_1$  are $1+O(x)$ times the corresponding coefficients of $T_0$. $T_\infty$ is called a $O(x^\infty)$ operator, satisfying $T_\infty (u)= O(x^\infty)$, where $u$ is a solution of \eqref{eqn:main}. If $T_0$ is of another form, we can define $T_0+O(x)$ in a similar way.

\begin{proposition}\label{lem:Lap_cx}
 When acting on $u$ or function $\psi$ independent of $\theta$, we have
\begin{align*}
g^{\bar{j} i}\partial_i\partial_{\bar j}=\eta^{\bar{\beta}\alpha} \partial_\alpha\partial_{\bar{\beta}}+ x^2 Re( C_{\alpha \bar{n}}^*\frac{\partial^2}{\partial z_\alpha\partial x})+\frac{1}{2}x^2 \partial^2_x+ x\partial_x+O(x),
\end{align*}
  where
$
C^*_{\alpha \bar{n}} = C_{\alpha \bar{n}} (\cos \theta -\sqrt{-1} \sin \theta)
$. Here $C_{i\bar{j}}$ are bounded smooth in $x, \theta$ and other complex coordinates, and $(\eta^{\bar\beta\alpha})_{1\leq \alpha,\beta\leq n-1}$ is the inverse of the $(n-1)\times (n-1)$ matrix $(\vartheta_{\alpha\bar\beta}|_D)$. Note that $\eta^{\bar{\beta}\alpha} u_{\alpha \bar{\beta}}$ is just the Laplacian operator on $D$ with respect to the canonical KE metric. For simplicity, we set $\Delta_D:=\eta^{\bar{\beta}\alpha}\partial_\alpha\partial_{\bar\beta}$.
\end{proposition}

What are most relevant to us is another set of non-holomorphic coordinates. 
 By tubular neighborhood theorem, we can find a neighborhood $U_\delta$ of $D$ diffeomorphic to the normal bundle of $D$. Even though the normal bundle can be made to a complex manifold, this diffeomorphism is in general not holomorphic. The coordinates we use are bundle coordinates: Locally 
$z^*_\alpha$ ($\alpha=1,\dots n-1$) are coordinates of $D$, and $\xi,\eta$ are fiber coordinates. We also need polar coordinates with respect to $(\xi,\eta)$, namely $\xi=x^* \cos\theta^*, \eta= x^* \sin\theta^*$. Again, we set $x^*=1/(-\log (x^*)^2).$  Note that $x^*$ is globally defined on $U_\delta$.

We have the following relations between $(z_1,\dots, z_n)$ and $(z^*_1,\dots, z^*_{n-1}, \xi,\eta)$:
\begin{enumerate}
\item $z^*_\alpha |_D=z_\alpha |_D$ for $\alpha=1,\dots, n-1$.
\item $z_n=0$ if and only if $\xi=\eta=0$.
\item $x^*$ equals the distance to $D$ with respect to some fixed good Riemannian metric, and hence is a globally defined function on $U_\delta$.
\end{enumerate}

We shall need the following three technical lemmas, whose proofs are contained in the Appendix.

\begin{lemma}\label{lem-theta*}
All the derivatives of $u$ with respect to $\theta^*$ are of order $O(x^\infty)$.
\end{lemma}

\begin{lemma}\label{lem-linear*}
 We also have
\begin{align}
 \label{eq-TrV-XStar} g^{\bar ji}\partial_i\partial_{\bar j}={1\over 2}(x^*)^2\frac{\partial^2}{(\partial x^*)^2}+x^*\frac{\partial}{\partial x^*}+(x^*)^2 Re( C_{\alpha \bar{n}}^*\frac{\partial^2 }{\partial x \partial z^*_\alpha})+\Delta_D +O(x^*),
\end{align}
when acting on $u$ or on $\psi$ which is independent of $\theta^*$.
\end{lemma}

\begin{lemma}\label{lem-deriv} 
For any $i\in I$ and integer $j\geq 0$, and smooth function $c_{i,j}$ on $D$, we have
$$|\nabla\bar\nabla (c_{i,j}(x^*)^i(\log x^*)^j)|_\omega=O((x^*)^i(-\log x^*)^j).$$
\end{lemma}

In the rest of this paper, locally we always use coordinate charts like
$z_\alpha^*, x^*, \theta^*$. By for simplicity, we still denote them as  $z_\alpha, x, \theta.$

{\bf To simplify notations, in the following sections of the paper, we simply write $(x^*, \theta^*)$ as $(x,\theta)$. Hopefully this will cause no confusion.}

Given a function $\psi$ in $U_\delta$, we can average the $\theta$ direction by integration\footnote{This kind of construction has already appeared in \cite{RoZh} }. To be precise, for any fixed point $p\in D$ with coordinate $z'$, and fixed $x$, we derive a function
\begin{align*}
\tilde{\psi}( z', x) = \frac{1}{2\pi} \int_0^{2\pi}  \psi(z^\prime, x, \theta) d\theta   ,
\end{align*}
under the local coordinate chart. $\tilde{\psi}$ is globally defined, and according to Lemma \ref{lem:theta},
\begin{align*}
R(z^\prime, x, \theta):=\psi(z^\prime, x, \theta)-\tilde{\psi}(z^\prime, x) = O(x^\infty).
\end{align*}
For example, we have by Lemma \ref{lem-f} that 
\begin{align}\label{eq-f-tilde}
f=\sum_{i=0}^\infty \tilde f_i x^i+O(x^\infty).
\end{align}
{\bf In the following sections, a $``\ \tilde{}\ "$ over a function always means its average on the $\theta$-direction.}

\section{Constructing a formal expansion}\label{formal}

Lemma \ref{lem:theta} suggests that we should find approximate solutions of the form
\begin{align*}
\psi_k= \sum_{i\in I, i\leq k}\sum_{j=0}^{N_i} c_{i,j} x^i (\log x)^j
,\end{align*}
where $I$ is an index set defined below, and
$c_{i,j}$ are functions on $D$, such that,
\begin{align*}
Q(\psi_k):=\log\frac{(\omega +\ddbar \psi_k)^n}{ \omega^n}-\psi_k-f=O(x^{k_+}(\log x)^{N_{k_+}}).
\end{align*}
Here integers $N_{i}$ can be defined inductively and explicitly if $I$ is known. $k_+$ is the next larger element of $k$ in $I$. Note that by Lemma \ref{lem-deriv}, $\omega +\ddbar \psi_k$ is uniformly equivalent to $\omega$ when $x$ is small.

The index set $I$ is defined as follows: First we assume $\{\lambda_k\}$'s are increasing eigenvalues of $-\Delta_D=\eta^{\bar{\beta}\alpha}\partial_\alpha\partial_{\bar\beta}$ on $D$.
Denote the two zeros of $\frac{1}{2}k^2+\frac{1}{2}k-1-\lambda_k$ by $\overline{m}_k, \underline{m}_k$,  where
$
\overline{m}_k\geq 1,
\underline{m}_k \leq -2,
$
and
\begin{align}\label{eq-mk}
\overline{m}_k \sim \sqrt{2\lambda_k}, \quad
\underline{m}_k \sim -\sqrt{2\lambda_k}
\end{align}
as $k \rightarrow \infty$.
Then we define the index set $I$ as the monoid generated by $\{1\}\cup\{\overline{m}_k\}_{k=1}^\infty$, and align its elements in the ascending order.
We denote $E_\lambda$ the eigenfunction space of $-\Delta_D$ with respect to the eigenvalue $\lambda$, and $E_\lambda^\perp$ its perpendicular space.

We need to approximate the operator $Q(\psi)$ by its linearization and estimate its error. The following calculation is well-known:
\begin{lemma}\label{lem-taylor} For any smooth function $\psi$ defined near $D$ such that $\omega+\ddbar\psi>0$ and is equivalent to $\omega$, then we have
$$|\log\frac{(\omega+\ddbar\psi)^n}{\omega^n}-g^{\bar j i}\psi_{i\bar j}|\leq C|\nabla\bar\nabla\psi|_\omega^2.$$
\end{lemma}

\begin{proof}
Write $g_t$ for the metric associated to $\omega+t\ddbar\psi$, then we have
\beqs
\log\frac{(\omega+\ddbar\psi)^n}{\omega^n}&=&\int_0^1\frac{\partial}{\partial t}\log\det(g_{i\bar j}+t\psi_{i\bar j}) dt= \Big(\int_0^1 g_t^{\bar j i} dt\Big) \psi_{i\bar j}\\
&=& g^{\bar j i}\psi_{i\bar j}+\Big(\int_0^1 g_t^{\bar j i}-g^{\bar j i} dt\Big) \psi_{i\bar j}\\
&=&g^{\bar j i}\psi_{i\bar j}+\Big(\int_0^1\int_0^1\frac{\partial}{\partial s} g_{st}^{\bar j i}ds\ dt\Big) \psi_{i\bar j}\\
&=& g^{\bar j i}\psi_{i\bar j}-\Big(\int_0^1t \big(\int_0^1 g_{st}^{\bar j p}g_{st}^{\bar qi }ds\big)dt\Big) \psi_{i\bar j}\psi_{p\bar q}.
\eeqs
By our assumption, the metrics $g_{st}$ are uniformly equivalent to $g$, so we get the conclusion from the above identity.
\qed
\end{proof}

As the 0th order approximation, we choose $\psi_0=c_0$, then by Lemma \ref{lem-f}, we have 
$$Q(\psi_0)=O(x).$$
To find higher order approximations, we define:
$$ N:=\frac{1}{2}x^2\frac{\partial^2}{\partial x^2}+x\frac{\partial}{\partial x}-1.$$
If $\psi_1:=c_0+c_{1,0} x$, then
\beqs
Q(\psi_1)&=&g^{\bar ji}\partial_i\partial_{\bar j}\psi_1+O(|\nabla\bar\nabla\psi_1|^2_\omega)-\psi_1-f\\
&=& \Delta_D\psi_1+N\psi_1-f+O(x^2)\\
&=& (\Delta_D c_{1,0}-\tilde f_1)x+O(x^2),
\eeqs
where $\tilde{f}_1$ is defined in \eqref{eq-f-tilde}.
However, $\Delta_D c_{1,0}=\tilde f_1$ is solvable if and only if $\tilde f_1$ is orthogonal to the eigenfunctions associate to the 0 eigenvalue of $\Delta_D$, i.e. $\int_D \tilde f_1 dv_D=0$.
If this is not true, we shall need a log-term correction: Set instead $\psi_1:=c_0+c_{1,0} x+c_{1,1}x\log x$, then 
\beqs
Q(\psi_1)&=&g^{\bar ji}\partial_i\partial_{\bar j}\psi_1+O(|\nabla\bar\nabla\psi_1|^2_\omega)-\psi_1-f\\
&=& \Delta_D\psi_1+N\psi_1-f+O(x^2(\log x)^2)\\
&=& (\Delta_D c_{1,0}+\frac{3}{2}c_{1,1}-\tilde f_1)x+(\Delta_D c_{1,1}) x\log x+O(x^2(\log x)^2).
\eeqs
If we require $Q(\psi_1)=O(x^2(\log x)^2)$, then we have
$$\Delta_D c_{1,0}+\frac{3}{2}c_{1,1}-\tilde f_1=0,\quad \Delta_D c_{1,1}=0.$$
So $c_{1,1}$ must be a constant such that $\int_D (\frac{3}{2}c_{1,1}-\tilde f_1)dv_D=0$. Then $c_{1,0}$ is solvable and unique up to a constant. So $c_{1,0}$ can not be determined locally, hence we call it ``the first global term".

\begin{lemma}\label{lem-c11}
The coefficient $c_{1,1}$ is a topological number, depending only on $D$ and its normal bundle.
\end{lemma} 

\begin{proof} By the previous discussion, we have
$$c_{1,1}=\frac{2}{3}\fint_D \tilde f_1|_D dv_D.$$
Recall that
$$f=\log\frac{\Phi}{\|s\|^2(\log\|s\|^2)^2\omega^n}.$$
In local holomorphic coordinates, we have
 $$\det(g_{i\bar j})=\det(\tilde \vartheta_{i\bar j}+N_iN_{\bar j})=\det(\tilde\vartheta_{i\bar j})(1+\tilde\vartheta^{\bar j i}N_iN_{\bar j}).$$
 A direct computation shows that
\beqs
\tilde f_1|_D&=& (\partial_{x^*} f)|_D= (\partial_x f)|_D\\
&=& \partial_x|_{x=0}\log\frac{\phi}{\|s\|^2(\log\|s\|^2)^2 \det(g_{i\bar j})}\\
&=& 2\Big(\vartheta^{\bar q p}-\frac{\vartheta^{\bar n p}\vartheta^{\bar q n}}{\vartheta^{\bar n n}}\Big)|_D\cdot \varphi_{p\bar q}|_D\\
&=& 2\Big(\vartheta^{\bar \beta \alpha}-\frac{\vartheta^{\bar n \alpha}\vartheta^{\bar \beta n}}{\vartheta^{\bar n n}}\Big)|_D \cdot \varphi_{\alpha\bar \beta}|_D\\
&=&2\eta^{\bar\beta\alpha}\varphi_{\alpha\bar\beta}|_D.
\eeqs
It is essentially the trace of $(\ddbar\varphi)|_D$ with respect to $\omega_D$, up to a constant factor. Since the restriction of the line bundle $\cO(D)$ to $D$ is just the normal bundle of $D$ in $\overline X$, denoted by $N_D$, and $[\omega_D]=2\pi c_1(K_D)$, we have $c_{1,1}$ equals 
$$\frac{K^{n-2}_D\cdot N_D}{K^{n-1}_D}$$
up to a constant factor depending only on $n$. \qed
\end{proof}

Note that the appearance of $x\log x$ term and its coefficients are already pointed out by Rochon-Zhang \cite{RoZh}.

Now we proceed to higher order approximations:  Suppose we have already find $\psi_-$ such that 
$$Q(\psi_-)=d_{i,j} x^i(\log x)^j +O(x^i(\log x)^{j-1} ),$$
where $i\in I$ and $d_{i,j}$ is a smooth function on $D$.
At present we assume $j>0$. We want to find $c_{i,j}\in C^\infty(D)$ such that for $\psi:=\psi_-+c_{i,j}x^i(\log x)^j$, we have
$$Q(\psi)=O(x^i(\log x)^{j-1}).$$
Now we have
\beqs
Q(\psi)&=& Q(\psi_-)+\log\frac{\big(\omega_{\psi_-}+\ddbar(c_{i,j}x^i(\log x)^j)\big)^n}{\omega^n_{\psi_-}}-c_{i,j}x^i(\log x)^j\\
&=& (d_{i,j}-c_{i,j})x^i(\log x)^j+g^{\bar q p}_{\psi_-}\partial_p\partial_{\bar q}(c_{i,j}x^i(\log x)^j)+O(x^i(\log x)^{j-1}),
\eeqs
where we use Lemma \ref{lem-taylor} in the second equality. Since 
$$g^{\bar q p}_{\psi_-}-g^{\bar q p}=-\Big(\int_0^1 g^{\bar q k}_{t\psi_-}g^{\bar l p}_{t\psi_-}dt \Big)\psi_{k\bar l}$$
and $|\nabla\bar\nabla\psi|_\omega =O(x\log x)$, by Lemma \ref{lem-deriv} we have
\beqs
Q(\psi)&=& (d_{i,j}-c_{i,j})x^i(\log x)^j+g^{\bar q p}\partial_p\partial_{\bar q}(c_{i,j}x^i(\log x)^j)+O(x^i(\log x)^{j-1})\\
&=& \big(d_{i,j}+(\Delta_D+\frac{1}{2}i^2+\frac{1}{2}i-1)c_{i,j}\big)x^i(\log x)^j+O(x^i(\log x)^{j-1}).
\eeqs
If $\frac{1}{2}i^2+\frac{1}{2}i-1$ is not an eigenvalue of $-\Delta_D$, or equivalently $i$ is not one of the $\overline m_k$'s, then we can find a unique $c_{i,j}$ such that $Q(\psi)=O(x^i(\log x)^{j-1})$.

If $i=\overline m_l$ for some $l\in\N$,  write $d_{i,j}=d_{i,j}^0+d_{i,j}^\perp$, the orthogonal decomposition with respect to $E_{\lambda_l}$. We need to first modify $\psi_-$:
 
\noindent{\bf Claim:} There is a smooth function $\rho\in C^\infty(D)$ such that
$$Q(\psi_-+\rho x^i(\log x)^{j+1})=d_{i,j}^\perp x^i(\log x)^j+O(x^i(\log x)^{j-1}).$$
In fact, the same computation as above gives
\beqs
Q(\psi_-+\rho x^i(\log x)^{j+1})&=& d_{i,j}x^i(\log x)^j+g^{\bar q p}\partial_p\partial_{\bar q}(\rho x^i(\log x)^{j+1})\\
& &-\rho x^i(\log x)^{j+1}+O(x^i(\log x)^{j-1})\\
&=&\big(d_{i,j}+(j+1)(i+\frac{1}{2})\rho\big)x^i(\log x)^j\\
& & +\big(\Delta_D \rho +\lambda_l\rho\big)x^i(\log x)^{j+1}+O(x^i(\log x)^{j-1}).
\eeqs
We can simply choose $\rho$ to be a constant multiple of $d_{i,j}^0$. Then we can find $c_{i,j}$ such that 
$$\psi:=\psi_-+c_{i,j}x^i(\log x)^j+\rho x^i(\log x)^{j+1}$$
satisfies
$$Q(\psi)=O(x^i(\log x)^{j-1}).$$
Note that in this case $c_{i,j}$ is unique up to an element of $E_{\lambda_l}$.

When $j=0$, initially we have $Q(\psi_-)=d_{i,0} x^i+O(x^{i_+}(\log x)^m)$, where $i_+$ is the next term of $i$ in $I$ and $m$ depends on the choice of $\psi_-$. We try to find $\psi=\psi_-+c_{i,0}x^i+\rho x^i\log x$ such that $Q(\psi)=O(x^{i_+}(\log x)^m)$. The discussion is the same as above.

\begin{remark}
From the above discussion, we see that  $c_{\overline{m}_l, 0} $'s are all independent global terms. For any $l\in \mathbb{N}$, $c_{\overline{m}_l, 0}$ is unique up to an element in $E_{\lambda_l}$. So we have infinitely many formal solutions. There exists special formal solutions such that $x^{\overline m_l}(\log x)^j$ appears in the formal solution only if $\overline m_l\in\N$. However, from our proof in the next section, other non-integer $\overline{m}_l$'s also appear in the expansion in general.
\end{remark}

\section{Proof of Theorem \ref{main}}\label{iterate}

In this section, we prove Theorem \ref{main} by induction. We shall prove that if we have an asymptotic expansion of certain order, we can obtain a higher order expansion. The main tools are the solution formula for second order ODE's with constant coefficients and the method of ``separation of variables". We shall use the fact that if a function $\psi$ on $D$ has better regularity, then the ``generalized Fourier series" of $\psi$ with respect to the eigenfunctions of $\Delta_D$ has better convergence properties. 

We write $u$ as $u= c_0+ \tilde{v}+ R$, where
\begin{align*}
R(x, z^\prime, \theta):&= u(x, z^\prime, \theta)- \frac{1}{2\pi} \int_{S^1} u(x, z^\prime, \theta_1) d\theta_1,\\
&= \frac{1}{2\pi} \int_{S^1}\left( \int_0^1\partial_\theta u(x, z^\prime, t\theta+(1-t)\theta_1)dt\right) (\theta-\theta_1)d\theta_1
\end{align*}
which is
 $\bold{O}(x^\infty)$ by Lemma \ref{lem:theta}, and
\begin{align*}
\tilde{v}(x, z^\prime)= \frac{1}{2\pi}  \int_{S^1}u(x, z^\prime, \theta) d\theta - c_0.
\end{align*} 

Define a nonlinear differential operator:
$$F(\psi):=N\psi+\Delta_D \psi-Q(\psi).$$
Then the equation \eqref{eqn:main} becomes
$$N u+\Delta_D u=F(u).$$
Taking average with respect to $\theta$ in both sides, we have
\begin{align}\label{eq-av}
N\tilde{v} + \Delta_D \tilde{v} = \tilde{F}(\tilde{v}),
\end{align}
where
\begin{align*}
\tilde{F}(\tilde{v}) &=\frac{1}{2\pi}  \int_0^{2\pi}  \left( F(c_0+\tilde{v}+R)+c_0 \right) d \theta\\
&=\frac{1}{2\pi}\int_0^{2\pi}     \left( F(c_0+\tilde{v})+c_0 \right) d  \theta+\bold{O}(x^\infty) .
\end{align*}

We can view $\tilde v$ as a function on $D\times [0,\delta)$. Even though $F$ depends on the choice of local coordinates, $\tilde F(\tilde v)$ is globally defined since the left hand side of \eqref{eq-av} is globally defined on $D\times [0,\delta)$.

\begin{definition}\label{def-exp}
We say $\tilde{v}$ has an expansion of order $O(x^k)$, for some $k \in I$, if there are smooth coefficients $c_{i, j}$ defined on $D$ such that,
\begin{align*}
\tilde{v}=\psi_k+R_k= \sum_{ i\in I, i\leq k} \sum_{j=0}^{N_i} c_{i, j}x^i (\log x)^j+R_k,
\end{align*}
where the remainder $R_k$ satisfies, under the local coordinate system, for some $\epsilon\in (0, k_+-k)$,
\begin{align*}|
x^l \partial^l_x \partial^m_{z',\bar z'} R_k|\leq C(k, l,m,\epsilon) x^{k+\epsilon},
\end{align*}
 for any integers $l, m$. Equivalently, $R_k=\bold{O}(x^{k+\epsilon})$.
\end{definition}

According to Lemma \ref{lem-u-c0} and \ref{lem-v_higher}, $\tilde{v}= 0+ R_0$, where $R_0=\tilde{v}$ such that $|R_0|_{C^l_g}\leq  C_l x(-\log x)$ for any $l$.
So we say that $\tilde{v}$ has an expansion of order $O(x^0)$.

Inductively, we assume $\tilde{v}$ has an expansion of order $k$, and the goal is to prove for case $k_+$.

We adjust $\epsilon$ smaller if necessary, such that $\overline{m}_l> k_+$ implie that $\overline{m}_l> k_++\epsilon$. Indeed, we do not have a uniform $\epsilon$ for all $l$.

\begin{lemma}
If $\tilde{v}$ has an expansion of order $O(x^k)$, say, $\tilde v=\psi_k+R_k$, then  $\psi_k$ coincides with one of the formal approximate solutions constructed in $\S$ \ref{formal},  and
$\tilde{F}(\tilde{v})$ has an expansion of order $O(x^{k_+})$, i.e.,
\begin{align}\label{eq-FTilde}
\tilde{F}(\tilde{v})= \sum_{i\in I, i\leq k_+} \sum_{j=0}^{N_i} \tilde{F}_{i, j} x^i (\log x)^j +\tilde{R}_{F, k_+}
\end{align}
with $\tilde{R}_{F, k_+}=\bold{O}(x^{k_++\epsilon})$. 
\end{lemma}

\begin{proof}  We first prove that if $\tilde v=\psi_k+R_k$, then $\psi_k$ coincides with one of the formal approximate solutions. By the discussion in \S \ref{formal}, it is easy to see that we only need to prove
$$Q(c_0+\psi_k)=\bold{O}(x^{k +\epsilon}),$$
then by the induction argument in \S \ref{formal}, $\psi_k$ must coincide with one of the formal approximate solutions for any $k\in \Z_{\geq 0}$.

In fact, since $u=c_0+\psi_k+R_k+R$ satisfies $Q(u)=0$, we must have
\beqs
0&=&Q(c_0+\psi_k+R_k+R)=Q(c_0+\psi_k)+\int_0^1 \frac{d}{dt} Q(c_0+\psi_k+t(R_k+R)) dt\\
&=& Q(c_0+\psi_k)+\int_0^1 g^{\bar j i}_{\psi_k+t(R_k+R)} \partial_i\partial_{\bar j}(R_k+R) dt-R_k-R.
\eeqs
Since $R_k+R=\bold{O}(x^{k+\epsilon})$, from the above equation we get $Q(c_0+\psi_k)=\bold{O}(x^{k +\epsilon}).$

Now we  prove that $\tilde F(\tilde v)$ has an expansion of order $O(x^{k_+})$ if $\tilde v$ has an expansion of order $O(x^k)$. For this, by definition, we only need to check the expansion for $F(c_0+\tilde v)$. 

If $k=0$, $F(c_0+\tilde{v})=F(c_0+R_0)=f+c_0+ O(x^2(\log x)^2)$, which confirms the claim. And if $k\geq 1$, 
\beqs
F(c_0+\psi_k+R_k)&=& \Delta_D (\psi_k+R_k)+N(c_0+\psi_k+R_k)\\
& & -Q(c_0+\psi_k)-(Q(c_0+\psi_k+R_k)-Q(c_0+\psi_k))\\
&=& F(c_0+\psi_k) +(\Delta_D+N+1)R_k-\int_0^1 g^{\bar j i}_{\psi_k+tR_k}\partial_i\partial_{\bar j} R_k dt\\
&=& F(c_0+\psi_k) +(\Delta_D+N+1)R_k-g^{\bar j i}\partial_i\partial_{\bar j} R_k+\bold{O}(x^{k_++\epsilon})\\
&=& F(c_0+\psi_k)+\bold{O}(x^{k_++\epsilon}),
\eeqs
where the last equality comes from Lemma \ref{lem-linear*}. 
After averaging in $\theta$, $\tilde{F}(\psi_k)$ has explicit expansion of any order.  
From the above equality, we have $\tilde F(\psi_k+R_k)-\tilde F(\psi_k)=\bold{O}(x^{k_++\epsilon})$. So we conclude that $\tilde F(\tilde v)$ has an expansion of order $O(x^{k_+})$. \qed 
\end{proof}

Assume $-\Delta_D \varphi_l = \lambda_l \varphi_l$ for analytic functions $\varphi_l$ on $D$ with $\int_D \varphi^2_l dv_D=1$, where $dv_D$ is the volume form associated to the canonical K\"ahler-Einstein metric $\omega_D$. Then $\{\varphi_l\}_{l=0}^\infty$ is an orthonormal basis of $L^2(D, dv_D)$.
As for any fixed $x>0$, $\tilde{v}$ is a smooth function on $D$,
 we write $\tilde{v}=\sum_l \tilde{v}_l(x) \varphi_l,$ where
\begin{align*}
\tilde{v}_l=\int_{D} (\tilde{v} \varphi_l)  dv_D,
\end{align*}

 Then
\begin{align}\label{eq-EquCoeff}
-\lambda_l \tilde{v}_l+N \tilde{v}_l=(\tilde{F}(\tilde{v}))_l,
\end{align}
where
\begin{align*}
(\tilde{F}(\tilde{v}))_l(x)=\int_{D} \tilde{F}(\tilde{v}) \varphi_l  dv_D.
\end{align*}
We view \eqref{eq-EquCoeff} as a non-homogeneous second order ordinary differential equation with respect to $x$, then we solve out by choosing a fixed small $x_0>0$:
\begin{align}\begin{split} \label{eq-ODESolu}
\tilde{v}_l=& C_1x^{\overline{m}_l}+C_2 x^{\underline{m}_l}-\frac{ 2x^{\overline{m}_l} }{\overline{m}_l-\underline{m}_l}  \int_{x}^{x_0} (\tilde{F}(\tilde{v}))_l x^{-1-\overline{m}_l} dx\\
&\qquad+\frac{ 2x^{\underline{m}_l} }{\overline{m}_l-\underline{m}_l}  \int_{x}^{x_0} (\tilde{F}(\tilde{v}))_l x^{-1-\underline{m}_l} dx.
\end{split}
\end{align}
To determine $C_1, C_2$, we first take $x=x_0$ to get
\begin{align*}
\tilde{v}_l(x_0)= C_1 x_0^{\overline{m}_l}+C_2 x_0^{\underline{m}_l}.
\end{align*}
Since $\tilde{v}_l \in L^\infty$, we then multiply \eqref{eq-ODESolu} by $x^{-\underline{m}_l}$, and let $x \rightarrow 0$,
\begin{align*}
0=C_2+ \frac{2 }{\overline{m}_l-\underline{m}_l}  \int^{x_0}_{0}    (\tilde{F}(\tilde{v}))_l x^{-1-\underline{m}_l}    dx.
\end{align*}
We conclude,
\begin{align}\begin{split}\label{eq-ODESolu1}
\tilde{v}_l &= \Big(\tilde{v}_l (x_0)\cdot x_0^{-\overline{m}_l}+
\frac{2 x_0^{\underline{m}_l-\overline{m}_l }}{\overline{m}_l-\underline{m}_l} 
\int_0^{x_0}   (\tilde{F}(\tilde{v}))_l x^{-1-\underline{m}_l}   dx \Big)\cdot x^{\overline{m}_l}\\
&\qquad-\frac{ 2 x^{\overline{m}_l} }{\overline{m}_l-\underline{m}_l}  \int_{x}^{x_0}   (\tilde{F}(\tilde{v}))_l x^{-1-\overline{m}_l}   dx\\
&\qquad-\frac{ 2 x^{\underline{m}_l} }{\overline{m}_l-\underline{m}_l}  \int_{0}^{x}    (\tilde{F}(\tilde{v}))_l x^{-1-\underline{m}_l}   dx.
\end{split}
\end{align}

Then
\begin{align}\begin{split}\label{eq-ODESolu1}
\tilde{v} &= \sum_l \frac{ \tilde{v}_l (x_0)x^{\overline{m}_l} \varphi_l  }{ x_0^{\overline{m}_l}}+ \sum_l
\frac{2 x_0^{\underline{m}_l-\overline{m}_l }x^{\overline{m}_l} \varphi_l }{\overline{m}_l-\underline{m}_l} 
\int_0^{x_0}   (\tilde{F}(\tilde{v}))_l x^{-1-\underline{m}_l}   dx \\
&\qquad- \sum_l \frac{ 2 x^{\overline{m}_l}\varphi_l }{\overline{m}_l-\underline{m}_l}  \int_{x}^{x_0}   (\tilde{F}(\tilde{v}))_l x^{-1-\overline{m}_l}   dx\\
&\qquad-\sum_l \frac{ 2 x^{\underline{m}_l}\varphi_l }{\overline{m}_l-\underline{m}_l}  \int_{0}^{x}    (\tilde{F}(\tilde{v}))_l x^{-1-\underline{m}_l}   dx.
\end{split}
\end{align}

Fix an index $A>0$. The first summation of \eqref{eq-ODESolu1} is easy to treat:
\begin{align}\label{eq-non-int}
 \sum_l \frac{ \tilde{v}_l (x_0)x^{\overline{m}_l} \varphi_l  }{ x_0^{\overline{m}_l}}= \sum_{\overline{m}_l<A} \left(\frac{ \tilde{v}_l (x_0) \varphi_l  }{ x_0^{\overline{m}_l}}\right)\cdot x^{\overline{m}_l} +H_{A, 0}x^A,
\end{align}
where 
\begin{align*}
H_{A, 0}= \sum_{\overline{m}_l\geq A} \frac{ \tilde{v}_l (x_0)x^{\overline{m}_l-A} \varphi_l  }{ x_0^{\overline{m}_l}}. 
\end{align*}
This is an expansion with respect to $x$, and it only has finitely many terms with $\overline{m}_l < A$, and the corresponding coefficients only depend on $z^\prime$. Our goal is to estimate the derivatives of the remainder of the form
\begin{align*}
(x\partial_x)^p  \partial_{z^\prime, \bar{z}^\prime}^q H_{A, 0}.
\end{align*}
To this end, notice that for any $N\in \mathbb{N}$,
\begin{align}\begin{split}\label{eq-vl}
|\tilde v_l(x_0)| &= \frac{1}{\lambda_l^N} \Big|\int_D \tilde v(\cdot ,x_0)\cdot \Delta_D^N \varphi_l (\cdot)dv_D \Big|\\
&=\frac{1}{\lambda_l^N} \Big|\int_D (\Delta_D^N \tilde v(\cdot, x_0)\cdot \varphi_l(\cdot) dv_D\Big|\\
&\leq \frac{C(\tilde v, x_0, N)}{\lambda_l^N}.
\end{split}
\end{align}
Here $\Delta_D^N \tilde v(\cdot, x_0)$ is evaluated at $x=x_0$, so bounded by the interior estimates of $\tilde v$. Then for any $p, q\in \mathbb{N}$, by \eqref{eq-vl},
\begin{align*}
\left\lvert(x\partial_x)^p  \Delta_D^q  H_{A,0}\right\rvert 
&=\left\lvert(x\partial_x)^p  \Delta_D^q  \sum_l \frac{ \tilde{v}_l (x_0)x^{\overline{m}_l-A} \varphi_l  }{ x_0^{\overline{m}_l}}\right\rvert \\&= \sum_l \left\lvert (\overline{m}_l-A)^p\lambda_l^q  \frac{ \tilde{v}_l (x_0)x^{\overline{m}_l} \varphi_l  }{ x_0^{\overline{m}_l}}\right\rvert\\
&\leq  C(\tilde v, x_0, N)  \sum_l \left\lvert \lambda_l^{\frac{p}{2} + q-N}  \varphi_l  \right\rvert.
\end{align*}
By the standard estimates of eigenvalues and eigenfunctions of $\Delta_D$, (cf. for example \cite{Sog} Corollary 5.1.2), as $\dim_\mathbb{R} D =2n-2$, $\lambda_l \sim l^\frac{1}{n-1}$ as $l \rightarrow \infty$ and $|\lambda_l^{-\frac{2n-3}{4}} \varphi_l| \leq C( \omega_D)$. So we can set $N>\frac{p}{2}+q+\frac{2n-3}{4}$, so that
\begin{align*}
\sum_l \left\lvert \lambda_l^{\frac{p}{2} + q-N}  \varphi_l \right\rvert \leq C(N, \omega_D).
\end{align*}
Note that the interchange of $\Delta_D$ with infinite summation is justified by the above estimate.
Now $(x\partial_x)^p  \Delta_D^q  H_{A,0}$ is bounded, which further implies that $(x\partial_x)^p  \Delta_D^{q-1}  H_{A,0}$ is $C^{1,\alpha}_\omega(D)$ in $z^\prime, \bar{z}^\prime$ by standard elliptic estimates. Eventually it implies, (probably with a different $q$),
 \begin{align}\label{eq-non-int-estm}
 |(x\partial_x)^p  \partial_{z^\prime, \bar{z}^\prime}^q  H_{A,0}|\leq C(v, x_0, p, q, \omega_D).
 \end{align}

The main technical result of this section is the following proposition, whose proof follows the same line as above:

\begin{proposition} \label{thm-IntF} Fix an index $A>0$.
Assume that on $D\times (0, x_0]$, we have a function $F(x, z^\prime)=x^i (\log x)^j\cdot w(x, z^\prime)$ for some $i\in I, j\in \Z$, such that $i\leq A, 0\leq j\leq N_i$, and \begin{align*}
|x^l \partial^l_x \partial^m_{z',\bar z'} w|\leq C_{l, m},
\end{align*} 
under the local coordinate system. In addition, we assume that $w$ only depends on $z^\prime$ when $i<A$. 

Denote
\begin{align*}
F_l= \int_D F \cdot \varphi_l dv_D.
\end{align*}
 Then the following  terms
\begin{align*}
&H_1= \sum_{l=1}^\infty \frac{x_0^{\underline{m}_l-\overline{m}_l}  x^{\overline{m}_l} \varphi_l}{\overline{m}_l-\underline{m}_l}\int_0^{x_0} x^{-1-\underline{m}_l}F_l dx \\
&H_2= \sum_{l=1}^\infty \frac{x^{\underline{m}_l} \varphi_l}{\overline{m}_l-\underline{m}_l}\int_0^{x} x^{-1-\underline{m}_l}F_l dx \\
&H_3 =\sum_{l=1}^\infty  \frac{x^{\overline{m}_l} \varphi_l}{\overline{m}_l-\underline{m}_l}\int_x^{x_0} x^{-1-\overline{m}_l}F_l dx,
\end{align*}
have an expansion of the form:
\begin{itemize}
\item if $i=A$,
\begin{align}\label{eq-HExp1}
\sum_{l\in I, l<A}  H_{l, 0}(z^\prime) x^l  + \sum_{m=0}^{j+1} H_{A, m}(x, z^\prime) x^A(\log x)^m,
\end{align}
\item if $i<A$,
\begin{align}\begin{split}\label{eq-HExp2}
&\sum_{l\in I, l<i}  H_{l, 0}(z^\prime) x^l  + \sum_{m=0}^{j+1} H_{i, m}( z^\prime) x^i(\log x)^m\\
&\qquad+\sum_{l\in I, i<l<A}  H_{l, 0}(z^\prime) x^l  + \sum_{m=0}^{j+1} H_{A, m}(x, z^\prime) x^A,
\end{split}
\end{align}
\end{itemize}
where all the coefficients $H_{l, m}$'s satisfy for any $q\in  \mathbb{N}$,
\begin{align}\label{eq-HlmL2}
|\partial^q_{z',\bar z'} H_{l, m}|\leq C(l, m, i, j, q).
\end{align}
Here $H_{i, j+1}$ is not a zero function only if $i=\overline{m}_l$ for some $l \in \mathbb{N}$.

In addition, for any $p,q\in \mathbb{N}$, we have in $D \times
(0, {x_0}),$
\begin{align}\label{eq-Hlm-mp}
|x^p \partial^p_x \partial^q_{z',\bar z'} H_{l, m}|\leq C(l, m, i, j, p, q).
\end{align}
\end{proposition}

\begin{proof}
First we show the expansion and \eqref{eq-HlmL2}. Notice that $F_l$'s only depend on $x$.
For any integer $N>0$, $|\Delta_D^N F|\leq C(F, N)x^i (-\log x)^j$. So we have the estimate of generalized Fourier coefficients, if $\lambda_l\neq 0$,
\begin{align}\begin{split}\label{eq-Fl}
|F_l| &= \frac{1}{\lambda_l^N} \Big|\int_D F\cdot \Delta_D^N \varphi_l dv_D \Big|\\
&=\frac{1}{\lambda_l^N} \Big|\int_D (\Delta_D^N F) \varphi_l dv_D\Big|\\
&\leq \frac{C(F, N)}{\lambda_l^N}x^i (-\log x)^j.
\end{split}
\end{align}

Now we look into the three integrals, with  $H_1$ first. Formally $H_1$ is already in the form of \eqref{eq-HExp1} or \eqref{eq-HExp2}, as
\begin{align*}
H_1=  \sum_{\overline{m}_l <A}\left( \frac{x_0^{\underline{m}_l-\overline{m}_l}  \varphi_l}{\overline{m}_l-\underline{m}_l}\int_0^{x_0} x^{-1-\underline{m}_l}F_l dx\right) \cdot x^{\overline{m}_l}  + H_{A, 0}(z^\prime, x)x^A,
\end{align*}
where
\begin{align*}
H_{A, 0}= \sum_{\overline{m}_l \geq A} \frac{x_0^{\underline{m}_l-\overline{m}_l}  x^{\overline{m}_l-A} \varphi_l}{\overline{m}_l-\underline{m}_l}\int_0^{x_0} x^{-1-\underline{m}_l}F_l dx.
\end{align*}
Here $H_{A, 0}$ is an infinite summation. 
Again, as $\dim_\mathbb{R} D =2n-2$, we have $\lambda_l \sim l^\frac{1}{n-1}$ as $l \rightarrow \infty$ and $|\lambda_l^{-\frac{2n-3}{4}} \varphi_l| \leq C( \omega_D)$. 
Then 
 if  $\overline{m}_l\geq  A$, applying \eqref{eq-Fl},
\begin{align*}
\Big| \frac{x_0^{\underline{m}_l-\overline{m}_l}  x^{\overline{m}_l-A} \varphi_l}{\overline{m}_l-\underline{m}_l}\int_0^{x_0} x^{-1-\underline{m}_l}F_l dx \Big| &\leq  \lambda_l^{-N+\frac{2n-3}{4}} x_0^{\underline{m}_l-A} C(F) \cdot |\lambda_l^{-\frac{2n-3}{4}} \varphi_l|\\
&\leq C(F, N,  \omega_D, x_0) \lambda_l^{-N+\frac{2n-3}{4}}.
\end{align*}
We set $N=\frac{5}{2} n-\frac{3}{4}$, then
\begin{align*}
\sum_{\overline{m}_l\geq A}\lambda_l^{-N+\frac{2n-3}{4}}=\sum_{\overline{m}_l\geq A}\lambda_l^{-2n} \leq \sum_{\overline{m}_l\geq A}l^{-\frac{2n}{n-1}}
\end{align*}
which is bounded by a constant depending only on $n$. Hence
the sum of terms in $H_{A, 0}$  is convergent. In addition, for any $q \in\mathbb{N}$,
\begin{align*}
|\Delta_D^q H_{A, 0}| &= \left\lvert\sum_{\overline{m}_l \geq A} \frac{x_0^{\underline{m}_l-\overline{m}_l}  x^{\overline{m}_l-A} \lambda_l^p\varphi_l}{\overline{m}_l-\underline{m}_l}\int_0^{x_0} x^{-1-\underline{m}_l}F_l dx\right\rvert\\&\leq C(F, N, \omega_D, x_0)\sum_{\overline{m}_l\geq A}  \lambda_l^{p-N+\frac{2n-3}{4}}.
\end{align*}
We can set $N$ much larger than $q$, such that $\Delta_D^q H_{A, 0}$ is bounded, which further implies that $\Delta_D^{q-1} H_{A, 0}$ is $C^{1,\alpha}$ in $z^\prime, \bar{z}^\prime$. This eventually implies \eqref{eq-HlmL2}.

\bigskip
The discussion of $H_2$ is similar. 
The difference here is that all terms are of order $x^i (-\log x)^m$ for some $0\leq m\leq j$. 
\begin{itemize}

\item
if $i= A$, we write $H_2$ as $H_{A, j}x^A (\log x)^j$. We estimate $\Delta_D^q H_{A, j}$, by \eqref{eq-Fl},
\begin{align*}
& \left\lvert \Delta_D^q\left( \sum_{l=1}^\infty \frac{x^{\underline{m}_l-A}(\log x)^{-j} \varphi_l}{\overline{m}_l-\underline{m}_l}\int_0^{x} x^{-1-\underline{m}_l}F_l dx \right) \right\rvert\\
&\qquad = \left\lvert\sum_{l=1}^\infty \frac{x^{\underline{m}_l-A}(\log x)^{-j} \lambda_l^q\varphi_l}{\overline{m}_l-\underline{m}_l}\int_0^{x} x^{-1-\underline{m}_l}F_l dx  \right\rvert\\
&\qquad \leq   \sum_{l=1}^\infty \left\lvert  C(F, N) \lambda_l^{-N+q}\varphi_l \cdot\frac{x^{\underline{m}_l-A}(\log x)^{-j}}{\overline{m}_l-\underline{m}_l}\int_0^{x} x^{A-1-\underline{m}_l}(\log x)^j dx \right\rvert \\
&\qquad \leq  \sum_{i=1}^\infty  C(F, N, q, j)  \lambda_l^{-N+q}|\varphi_l|,
\end{align*}
which converges if $N$ is large comparing to $q$ and $n$.

\item
if $i<A$, by the assumption, $w$ only depends on $z^\prime$. Then
\begin{align*}
\int_0^{x} x^{-1-\underline{m}_l}F_l dx= \int_D w\varphi_l dv_D \cdot\int_0^{x} x^{i-1-\underline{m}_l}(\log x)^j dx,
\end{align*}
which generates terms like  $w_l\cdot  x^{i -\underline{m}_l}  (- \log x)^m$, for $0\leq m\leq j$. Hence we have the expansion as \eqref{eq-HExp2}. For the estimates of coefficients, we can proceed in a similar way as in case $i=A$ to derive \eqref{eq-HlmL2}.
\end{itemize}

\bigskip
For $H_3$, notice when $\overline{m}_l=i$, $\int x^{-1-\overline{m}_l} \cdot x^i (\log x)^j dx= \frac{1}{j+1}(\log x)^{j+1}$. So we may have a term of order $x^i (\log x)^{j+1}$ in the expansion of $H_3$. 
In $H_3$, for terms with $\overline{m}_l\geq A$, we just apply \eqref{eq-Fl} to show that
\begin{align}\label{eq-H3-mi>A}
\sum_{\overline{m}_l\geq A} \frac{x^{\overline{m}_l} \varphi_l}{\overline{m}_l-\underline{m}_l}\int_x^{x_0} x^{-1-\overline{m}_l}F_l dx
\end{align}
can be written as, 
\begin{itemize}
\item
if $i=A$, 
\begin{align*}
 H_{A, m}(z^\prime, x) x^A (\log x)^{m},
\end{align*}
where $m=j+1$ if $A=\overline{m}_l$ for some $l\in \mathbb{N}$; otherwise $m=j$ . 
\item if $i<A$, as $w$ only depends on $z^\prime$, \eqref{eq-H3-mi>A} can be written as
\begin{align*}
\sum_{m=0}^{j} H_{i, m}(z^\prime) x^i (\log x)^m + H_{A , 0}(z^\prime, x) x^A .
\end{align*}
\end{itemize}
All coefficients $H_{l,m}$ satisfy \eqref{eq-HlmL2}. 
For the finite terms with $\overline{m}_l<A$, (essentially we do not worry about finite summation),
\begin{itemize}

\item
if $i=A$,
\begin{align*}
\sum_{\overline{m}_l< A} \frac{x^{\overline{m}_l} \varphi_l}{\overline{m}_l-\underline{m}_l}\int_x^{x_0} x^{-1-\overline{m}_l}F_ldr &=\sum_{\overline{m}_l< A} \frac{x^{\overline{m}_l} \varphi_l}{\overline{m}_l-\underline{m}_l}\int_0^{x_0} x^{-1-\overline{m}_l}F_l dr\\
&\qquad \qquad - \sum_{\overline{m}_l< A} \frac{x^{\overline{m}_l} \phi_l}{\overline{m}_l-\underline{m}_l}\int_0^{x} x^{-1-\overline{m}_l}F_l dr,
\end{align*}
which can be dealt with in the same way as for $H_1$ and $H_2$, and
\item if $i<A$, $w$ only depends on $ z^\prime$, and we still derive \eqref{eq-HlmL2} by applying the explicit integral formula of $x^{-1-\overline{m}_l} \cdot x^i (\log x)^j$
and \eqref{eq-Fl}.
\end{itemize}
\bigskip

Secondly, we prove \eqref{eq-Hlm-mp}. 
As $H_{l, m}$ is independent of $x$ if $l< A$, so we only need to consider $(x\partial_x)^p H_{A, m}$. The only trouble is that $x\partial_x (x^{\overline{m}_l})= \overline{m}_l x^{\overline{m}_l}$, which  produces an extra  factor $\overline{m}_l \sim \lambda_l^{\frac{1}{2}}$. So we simply increase $N$ to deal with this factor. 
 \qed
\end{proof}

Now we continue the proof of Theorem \ref{main}.
Recall that $\tilde{F}(\tilde{v})$ has the expansion \eqref{eq-FTilde} by induction, and $\tilde{v}$ can be solved out from $\tilde{F}(\tilde{v})$ by \eqref{eq-ODESolu1}.

First apply \eqref{eq-non-int}, \eqref{eq-non-int-estm} with $A=k_++\epsilon$ to derive that $\sum_l \frac{ \tilde{v}_l (x_0)x^{\overline{m}_l} \varphi_l  }{ x_0^{\overline{m}_l}}$ has a boundary expansion of order $k_+.$ Here $\epsilon$ is well set such that $\overline{m}_l <A=k_++\epsilon$ implies that $\overline{m}_l\leq k_+$.

For each term $ \tilde{F}_{i, j} x^i (\log x)^j$ or $\tilde{R}_{F, k_+}$ in the expansion of \eqref{eq-ODESolu1}, applying Proposition \ref{thm-IntF} with $A= k_+ +\epsilon$ and $F= \tilde{F}_{i, j} x^i (\log x)^j, w= \tilde{F}_{i, j}$ or $F=\tilde{R}_{F, k_+} ,  w= x^{-k_+-\epsilon} \tilde{R}_{F, k_+}$ respectively, we derive finite many expansions in the form of \eqref{eq-HExp1}, \eqref{eq-HExp2}, with \eqref{eq-Hlm-mp} holds.

Summing up these finite many expansions, $v$ has a boundary expansion of order $O(x^{k_+})$ in the sense of Definition \ref{def-exp}.
Then we complete the induction.

\appendix

\section{Proof of Lemma \ref{lem-theta*}-Lemma \ref{lem-deriv}} 
\begin{proof}[Proof of Lemma \ref{lem-theta*}]
By chain rule, we have
$$\frac{\partial}{\partial \theta}=\frac{\partial z^*_\alpha}{\partial \theta}\frac{\partial}{\partial z^*_\alpha}+\frac{\partial \overline{z^*_\alpha}}{\partial \theta}\frac{\partial}{\partial \overline{z^*_\alpha}}+\frac{\partial x^*}{\partial \theta}\frac{\partial}{\partial x^*}+\frac{\partial \theta^*}{\partial \theta}\frac{\partial}{\partial \theta^*}.$$
It is easy to see that the coefficients of the first 3 terms are all of the order $O(r)$, while $\frac{\partial \theta^*}{\partial \theta}$ is of order $O(1)$ and non-vanishing near $D$. So we can prove by induction that derivatives with respect to $\theta^*$ are also of order $O(x^\infty)$. 
\qed
\end{proof}

\begin{proof}[Proof of Lemma \ref{lem-linear*}]
By Lemma \ref{lem:Lap_cx}, we first check
$$\frac{\partial}{\partial x}=\frac{\partial z^*_\alpha}{\partial x}\frac{\partial}{\partial z^*_\alpha}+\frac{\partial \overline{z^*_\alpha}}{\partial x}\frac{\partial}{\partial \overline{z^*_\alpha}}+\frac{\partial x^*}{\partial x}\frac{\partial}{\partial x^*}+\frac{\partial \theta^*}{\partial x}\frac{\partial}{\partial \theta^*}$$
 term by term:
$$\frac{\partial z^*_\alpha}{\partial x}=\frac{\partial z^*_\alpha}{\partial r}\frac{dr}{d x}=\frac{r}{2x^2}\frac{\partial z^*_\alpha}{\partial r}=O(x^\infty),$$
Similar for $\frac{\partial \overline{ z^*_\alpha}}{\partial x}$. 
$$\frac{\partial\theta^*}{\partial x}=\frac{\partial\theta^*}{\partial r}\frac{r}{2x^2}=\frac{r}{2x^2}Im(\frac{1}{z^*_n}\frac{\partial z^*_n}{\partial r})=O(x^{-2}).$$
By the previous Lemma, we can ignore all derivatives with respect to $\theta^*$, too. So we only need to compute $\frac{\partial x^*}{\partial x}$. Note that we have the Taylor expansion near $D$:
$$r^*=Ar+O(r^2)$$
where $A>0$ is a locally defined function on $D$. Then we have
$$\frac{\partial x^*}{\partial x}=\frac{dx^*}{dr^*}\frac{\partial r^*}{\partial r}\frac{dr}{dx}=\frac{r(x^*)^2}{x^* x^2}\frac{\partial r^*}{\partial r}=\frac{A+O(r)}{A+O(r)}(1+4\log A\ x+O(x^2))=1+4\log A\ x+O(x^2)),$$
and so
$$\frac{\partial^2 x^*}{\partial x^2}=4\log A+O(x)=4\log A+O(x^*).$$
Then we conclude that
$$({1\over 2}x^2\frac{\partial^2}{\partial x^2}+x\frac{\partial}{\partial x})v=({1\over 2}(x^*)^2\frac{\partial^2}{(\partial x^*)^2}+x^*\frac{\partial}{\partial x^*})v+O(x^*)$$

Secondly we check tangential directions,
$$\frac{\partial}{\partial z_\alpha}=\frac{\partial z^*_\beta}{\partial z_\alpha}\frac{\partial}{\partial z^*_\beta}+
\frac{\partial \bar z^*_\beta}{\partial z_\alpha}\frac{\partial}{\partial \bar z^*_\beta}
+\frac{\partial x^*}{\partial z_\alpha}\frac{\partial}{\partial x^*}+\frac{\partial \theta^*}{\partial z_\alpha}\frac{\partial}{\partial \theta^*}.$$
We compute term by term:

First we can write $z^*_\beta=z_\beta+a_\beta(z)$, where $a_\beta$ is a local smooth function of $z$ such that $a_\beta(z_1,\dots, z_{n-1},0)\equiv 0$. This implies 
$$\frac{\partial a_\beta}{\partial z_\alpha}=O(r),\quad \frac{\partial\bar a_\beta}{\partial z_\alpha}=O(r),\quad $$
so
$$\frac{\partial z^*_\beta}{\partial z_\alpha}=\delta_{\alpha\beta}+O(r),\quad \frac{\partial \bar z^*_\beta}{\partial z_\alpha}=O(r).$$
For similar reason, all the higher order purely tangential derivatives of $z^*_\alpha$ are all of order $O(r)$.
Second, we have
$$\frac{\partial x^*}{\partial z_\alpha}=\frac{2(x^*)^2}{r^*}\big(\cos\theta^*\frac{\partial\xi}{\partial z_\alpha}+\sin\theta^*\frac{\partial\eta}{\partial z_\alpha}\big).$$
Recall that $\xi|_D\equiv 0$, we conclude that $\frac{\partial\xi}{\partial z_\alpha}$ (and in fact all the tangential derivatives) is of order $O(x^*)$, from which we conclude that
$$\frac{\partial x^*}{\partial z_\alpha}=O((x^*)^2).$$
Similarly, 
$$\frac{\partial \theta^*}{\partial z_\alpha}=\frac{1}{r^*}\big(-\sin\theta^*\frac{\partial\xi}{\partial z_\alpha}+\cos\theta^*\frac{\partial\eta}{\partial z_\alpha}\big)=O(1).$$

From the expression of $\frac{\partial x^*}{\partial z_\alpha}$, we can further calculate $\frac{\partial^2 x^*}{\partial z_\alpha\partial\bar z_\beta}$ by chain rule, and easy  to see that all but one term are of order $O((x^*)^2)$. The remaining term is
$$-\frac{\partial x^*}{\partial z_\alpha}\frac{1}{r^*}\frac{\partial r^*}{\partial \bar z_\beta}=-\frac{\partial x^*}{\partial z_\alpha}\frac{1}{r^*}\big(\cos\theta^*\frac{\partial\xi}{\partial \bar z_\beta}+\sin\theta^*\frac{\partial\eta}{\partial \bar z_\beta}\big).$$
This is again of order $O((x^*)^2)$. For the same reason,
$$\frac{\partial^2 \theta^*}{\partial z_\alpha\partial\bar z_\beta}=O(1).$$

So we get
$$\frac{\partial}{\partial z_\alpha}=\frac{\partial}{\partial z^*_\alpha}+O((x^*)^2) \frac{\partial}{\partial x^*} +O((x^*)^\infty) \text{ operators}$$
For second order derivatives, we have 
\beqs
\frac{\partial^2}{\partial z_\alpha\partial\bar z_\beta}&=& 
\Big(\frac{\partial z^*_\mu}{\partial z_\alpha}\frac{\partial}{\partial z^*_\mu}+
\frac{\partial \bar z^*_\mu}{\partial z_\alpha}\frac{\partial}{\partial \bar z^*_\mu}
+\frac{\partial x^*}{\partial z_\alpha}\frac{\partial}{\partial x^*}+\frac{\partial \theta^*}{\partial z_\alpha}\frac{\partial}{\partial \theta^*}\Big)\circ\\
& &\Big(\frac{\partial z^*_\nu}{\partial \bar z_\beta}\frac{\partial}{\partial z^*_\nu}+
\frac{\partial \bar z^*_\nu}{\partial \bar z_\beta}\frac{\partial}{\partial \bar z^*_\nu}
+\frac{\partial x^*}{\partial \bar z_\beta}\frac{\partial}{\partial x^*}+\frac{\partial \theta^*}{\partial \bar z_\beta}\frac{\partial}{\partial \theta^*}\Big)\\
&=&
\frac{\partial \bar z^*_\nu}{\partial\bar z_\beta}\frac{\partial  z^*_\mu}{\partial z_\alpha} \frac{\partial^2}{\partial z^*_\mu\partial\bar z^*_\nu}+
\frac{\partial z^*_\nu}{\partial\bar z_\beta}\frac{\partial  z^*_\mu}{\partial z_\alpha} \frac{\partial^2}{\partial z^*_\mu\partial z^*_\nu}+
\frac{\partial z^*_\nu}{\partial\bar z_\beta}\frac{\partial  \bar z^*_\mu}{\partial z_\alpha} \frac{\partial^2}{\partial \bar z^*_\mu\partial z^*_\nu}+
\frac{\partial \bar z^*_\nu}{\partial\bar z_\beta}\frac{\partial \bar z^*_\mu}{\partial z_\alpha} \frac{\partial^2}{\partial \bar z^*_\mu\partial\bar z^*_\nu}
\\
& &+\frac{\partial \bar z^*_\nu}{\partial\bar z_\beta}\frac{\partial  x^*}{\partial z_\alpha} \frac{\partial^2}{\partial x^*\partial\bar z^*_\nu}
+\frac{\partial x^*}{\partial\bar z_\beta}\frac{\partial  z^*_\mu}{\partial z_\alpha} \frac{\partial^2}{\partial z^*_\mu\partial x^*}
+\frac{\partial x^*}{\partial\bar z_\beta}\frac{\partial \bar z^*_\mu}{\partial z_\alpha} \frac{\partial^2}{\partial \bar z^*_\mu\partial x^*}
+\frac{\partial x^*}{\partial z_\alpha}\frac{\partial  z^*_\nu}{\partial \bar z_\beta} \frac{\partial^2}{\partial z^*_\nu\partial x^*}\\
& &+\frac{\partial x^*}{\partial\bar z_\beta}\frac{\partial  x^*}{\partial z_\alpha} \frac{\partial^2}{(\partial x^*)^2}+
\frac{\partial^2 z^*_\nu}{\partial z_\alpha\partial \bar z_\beta}\frac{\partial}{\partial z^*_\nu}+
\frac{\partial^2 \bar z^*_\nu}{\partial z_\alpha \partial \bar z_\beta}\frac{\partial}{\partial \bar z^*_\nu}
+\frac{\partial^2 x^*}{\partial z_\alpha\partial \bar z_\beta}\frac{\partial}{\partial x^*}+\frac{\partial^2 \theta^*}{\partial z_\alpha\partial \bar z_\beta}\frac{\partial}{\partial \theta^*}\\
& &+\text{2nd order terms containing derivatives of}\ \theta^*\\
&=& \frac{\partial^2}{\partial z^*_\alpha\partial \bar z^*_\beta}+O((x^*)^2) \cdot \left((x^*)^2 \frac{\partial^2 }{(\partial x^*)^2}+ \frac{\partial }{\partial x^*} +\frac{\partial^2}{\partial x^*\partial\bar z^*_\beta}
+\frac{\partial^2}{\partial z^*_\alpha\partial x^*} \right)
\\
& &+O((x^*)^\infty) \text{ operators.}
\eeqs

Next we compute mixed derivatives.
$$\frac{\partial}{\partial z_n}=\frac{\partial z^*_\beta}{\partial z_n}\frac{\partial}{\partial z^*_\beta}+
\frac{\partial \bar z^*_\beta}{\partial z_n}\frac{\partial}{\partial \bar z^*_\beta}
+\frac{\partial x^*}{\partial z_n}\frac{\partial}{\partial x^*}+\frac{\partial \theta^*}{\partial z_n}\frac{\partial}{\partial \theta^*}.$$
It is obvious that
$$\frac{\partial z^*_\beta}{\partial z_n}=O(1),\quad \frac{\partial \bar z^*_\beta}{\partial z_n}=O(1).$$

$$\frac{\partial x^*}{\partial z_n}=\frac{2(x^*)^2}{r^*}\big(\cos\theta^*\frac{\partial\xi}{\partial z_n}+\sin\theta^*\frac{\partial\eta}{\partial z_n}\big)=O((x^*)^2(r^*)^{-1}).$$

$$\frac{\partial \theta^*}{\partial z_n}=\frac{1}{r^*}\big(-\sin\theta^*\frac{\partial\xi}{\partial z_n}+\cos\theta^*\frac{\partial\eta}{\partial z_n}\big)=O((r^*)^{-1}).$$
By chain rule, we can see that 
\beqs
z_n\frac{\partial^2 x^*}{\partial z_n\partial \bar z_\beta}&=& z_n\frac{2(x^*)^2}{r^*}\big(\cos\theta^*\frac{\partial^2\xi}{\partial z_n\partial\bar z_\beta}+\sin\theta^*\frac{\partial^2\eta}{\partial z_n\partial\bar z_\beta}\big)\\
& &+ z_n\frac{2(x^*)^2}{r^*}\big(-\sin\theta^*\frac{\partial\xi}{\partial z_n}+\cos\theta^*\frac{\partial\eta}{\partial z_n}\big)\frac{\partial \theta^*}{\partial\bar z_\beta}\\
& &+ \frac{2z_n}{r^*}\frac{\partial x^*}{\partial z_n}\frac{\partial x^*}{\partial\bar z_\beta}-\frac{z_n}{r^*}\frac{\partial x^*}{\partial z_n}\big(\cos\theta^*\frac{\partial\xi}{\partial \bar z_\beta}+\sin\theta^*\frac{\partial\eta}{\partial \bar z_\beta}\big)\\
&=& O((x^*)^2)+O((x^*)^2)+O((x^*)^3)+O((x^*)^2)=O((x^*)^2)
\eeqs
Similarly, we have
$$z_n\frac{\partial^2 \theta^*}{\partial z_n\partial \bar z_\beta}=O(1).$$
Since in the mixed second order derivatives, $\frac{\partial}{\partial z_n}$ always go with $z_n$, all derivatives involving $\theta^*$ can be ignored, we have
\beqs
z_n\frac{\partial^2}{\partial z_n\partial\bar z_\beta}&=& 
z_n\Big(\frac{\partial z^*_\mu}{\partial z_n}\frac{\partial}{\partial z^*_\mu}+
\frac{\partial \bar z^*_\mu}{\partial z_n}\frac{\partial}{\partial \bar z^*_\mu}
+\frac{\partial x^*}{\partial z_n}\frac{\partial}{\partial x^*}+\frac{\partial \theta^*}{\partial z_n}\frac{\partial}{\partial \theta^*}\Big)\circ\\
& &\Big(\frac{\partial z^*_\nu}{\partial \bar z_\beta}\frac{\partial}{\partial z^*_\nu}+
\frac{\partial \bar z^*_\nu}{\partial \bar z_\beta}\frac{\partial}{\partial \bar z^*_\nu}
+\frac{\partial x^*}{\partial \bar z_\beta}\frac{\partial}{\partial x^*}+\frac{\partial \theta^*}{\partial \bar z_\beta}\frac{\partial}{\partial \theta^*}\Big)\\
&=&
z_n\frac{\partial \bar z^*_\nu}{\partial\bar z_\beta}\frac{\partial  z^*_\mu}{\partial z_n} \frac{\partial^2}{\partial z^*_\mu\partial\bar z^*_\nu}+
z_n\frac{\partial z^*_\nu}{\partial\bar z_\beta}\frac{\partial  z^*_\mu}{\partial z_n} \frac{\partial^2}{\partial z^*_\mu\partial z^*_\nu}+
z_n\frac{\partial z^*_\nu}{\partial\bar z_\beta}\frac{\partial  \bar z^*_\mu}{\partial z_n} \frac{\partial^2}{\partial \bar z^*_\mu\partial z^*_\nu}\\
& &+z_n\frac{\partial \bar z^*_\nu}{\partial\bar z_\beta}\frac{\partial \bar z^*_\mu}{\partial z_n} \frac{\partial^2}{\partial \bar z^*_\mu\partial\bar z^*_\nu}
+z_n\frac{\partial \bar z^*_\nu}{\partial\bar z_\beta}\frac{\partial  x^*}{\partial z_n} \frac{\partial^2}{\partial x^*\partial\bar z^*_\nu}
+z_n\frac{\partial x^*}{\partial\bar z_\beta}\frac{\partial  z^*_\mu}{\partial z_n} \frac{\partial^2}{\partial z^*_\mu\partial x^*}\\
& &+z_n\frac{\partial x^*}{\partial\bar z_\beta}\frac{\partial \bar z^*_\mu}{\partial z_n} \frac{\partial^2}{\partial \bar z^*_\mu\partial x^*}
+z_n\frac{\partial x^*}{\partial z_n}\frac{\partial  z^*_\nu}{\partial \bar z_\beta} \frac{\partial^2}{\partial z^*_\nu\partial x^*}\
+z_n\frac{\partial x^*}{\partial\bar z_\beta}\frac{\partial  x^*}{\partial z_n} \frac{\partial^2}{(\partial x^*)^2}\\
& &+z_n\frac{\partial^2 z^*_\nu}{\partial z_n \partial \bar z_\beta}\frac{\partial}{\partial z^*_\nu}+
z_n\frac{\partial^2 \bar z^*_\nu}{\partial z_n\partial \bar z_\beta}\frac{\partial}{\partial \bar z^*_\nu}
+z_n\frac{\partial^2 x^*}{\partial z_n\partial \bar z_\beta}\frac{\partial}{\partial x^*}
+z_n\frac{\partial^2 \theta^*}{\partial z_n\partial \bar z_\beta}\frac{\partial}{\partial \theta^*}\\
& &+\text{2nd order terms containing derivatives of}\ \theta^*\\
&=& O((x^*)^2)\frac{\partial^2}{\partial x^*\partial \bar z_\beta}+O((x^*)^4)\frac{\partial^2}{(\partial x^*)^2}+O((x^*)^2)\frac{\partial}{\partial x^*}+O((x^*)^\infty) {\text{ operators}}
\eeqs

In sum,\eqref{eq-TrV-XStar} is verified. 
\qed
\end{proof}

\begin{proof}[Proof of Lemma \ref{lem-deriv}]
We use the local holomorphic coordinates to check this. First, we have
$$\frac{\partial^2}{\partial z_\alpha\partial\bar z_\beta} \Big(c_{i,j}(z_1^*,\dots, z_{n-1}^*) (x^*)^i(\log x^*)^j\Big)=\frac{\partial^2 c_{i,j}}{\partial z^*_\alpha\partial\bar z^*_\beta} (x^*)^i(\log x^*)^j+o((x^*)^i(\log x^*)^j).$$
Recall that
$$\frac{\partial x^*}{\partial z_\alpha}=O((x^*)^2), \quad \frac{\partial}{\partial z_\alpha}=\frac{\partial}{\partial z^*_\alpha}+O((x^*)^2) \frac{\partial}{\partial x^*} +O((x^*)^\infty) \text{ operators},$$
so we have
$$\frac{\partial^2}{\partial z_n\partial\bar z_\beta} \Big(c_{i,j}(z_1^*,\dots, z_{n-1}^*) (x^*)^i(\log x^*)^j\Big)=O((r^*)^{-1} (x^*)^{i+1}(\log x^*)^j).$$
Finally, we have
$$\frac{\partial^2}{\partial z_n\partial\bar z_n} \Big(c_{i,j}(z_1^*,\dots, z_{n-1}^*) (x^*)^i(\log x^*)^j\Big)=O( (x^*)^i(\log x^*)^j).$$
Then the lemma follows from (\ref{eq-GAB}).
\qed
\end{proof}

\bigskip

{\noindent \footnotesize $^{*}$ Department of Mathematics\\
Rutgers University, Piscataway, NJ 08854\\

\noindent $^{**}$ Department of Mathematics\\
Nanjing University, Nanjing, China 210093\\


\bigskip

\begin{thebibliography}{99}

\footnotesize

\bibitem{ACF1982CMP} L. Andersson, P. Chru\'sciel, H. Friedrich,
\emph{On the regularity of solutions to the Yamabe equation and the existence of
smooth hyperboloidal initial data for Einstein's field equations},
Comm. Math. Phys., 149(1992), 587-612.


\bibitem{Ban}S. Bando, {\em Einstein K\"ahler metrics of negative Ricci curvature on open K\"ahler manifolds}, K\"ahler metrics and moduli spaces, Adv. Stud. in Pure Math. 18 no.2 (1990), 105-136.

\bibitem{CG} J. Carlson and P. Griffiths, {\em A defect relation for equidimensional holomorphic mappings between algebraic varieties}, Ann. of Math. (2) 95 (1972), 557-584.

\bibitem{CY1} S. Cheng and S. Yau, {\em On the existence of a complete K\"ahler metric on non-compact complex manifolds and the regularity of Fefferman's equation}, Comm. Pure and Applied Math., 33(1980), 507-544.

\bibitem{CY2} S. Cheng and S. Yau, {\em Inequality between Chern numbers of singular K\"ahler surfaces and characterization of orbit space discrete group of SU(2,1)}, Contemporary Math., 49(1986), 31-44.

\bibitem{Fef} C. Fefferman, {\em Monge-Amp\`ere equations, the Bergman kernel, and geometry of pseudoconvex domains},  Ann. of Math. (2) 103 (1976), no. 2, 395-416.

\bibitem{Gra} R. Graham, {\em Higher asymptotics of the complex Monge-Amp\`ere equation}, Compositio Math. 64 (1987), no. 2, 133-155.

\bibitem{HJ1} Q. Han and X. Jiang, {\em Boundary expansions for minimal graphs in the hyperbolic space}, arXiv:1412.7608.

\bibitem{HJ2} Q. Han and X. Jiang, {\em Boundary expansions and convergence for complex Monge-Amp\`ere equations}, arxiv:1806.05371.

\bibitem{JMR} T. Jeffres, R. Mazzeo and Y. Rubinstein, {\em K\"ahler-Einstein metrics with edge singularities},  Ann. of Math. (2) 183 (2016), no. 1, 95-176. 


\bibitem{JW} H. Jian and X. Wang, {\em Bernstein theorem and regularity for a class of Monge-Amp\`ere equations},  J. Differential Geom. 93 (2013), no. 3, 431-469.


\bibitem{Ko} R. Kobayashi, {\em K\"ahler-Einstein metric on an open algebraic manifold}, Osaka J. Math., 21(1984), 399-418.


\bibitem{Le} J. Lee, {\em  Higher asymptotics of the complex Monge-Amp\`ere equation and the geometry of CR manifolds}, Thesis, Mass. Inst. of Tech. (1982).

\bibitem{LM}J. Lee and R. Melrose, {\em Boundary behaviour of the complex Monge-Amp\`ere equation}, Acta Math. 148(1982), 159-192.


\bibitem{LZ} J. Lott and Z. Zhang, {\em K\"ahler-Ricci flow on quasi-projective manifolds},  Duke Math. J. 156 (2011), no. 1, 87-123. 

\bibitem{RoZh} F. Rochon and Z. Zhang, {\em Asymptotics of complete K\"ahler metrics of finite volume on quasiprojective manifolds}, Adv. Math. 231(2012), 2892-2952.

\bibitem{San} B. Santoro, {\em On the asymptotic expansion of complete Ricci-flat K\"ahler metrics on quasi-projective manifolds}, J. Reine Angew. Math. 615 (2008), 59-91.

\bibitem{Sch}G. Schumacher, {\em Asymptotics of K\"ahler-Einstein metrics on quasi-projective manifolds and an extension theorem on holomorphic maps}, Math. Ann., 311(1998), 631-645.

\bibitem{Sog} C. Sogge, {\em Fourier integrals in classical analysis}, Cambridge Tracts in Mathematics, 1993.

\bibitem{SS} J. Sun and S. Sun, {\em Projective embedding of log Riemann surfaces and K-stability}, arXiv:1605.01089v3.


\bibitem{TY}G. Tian and S.-T. Yau, {\em Existence of K\"ahler-Einstein metrics on complete K\"ahler manifolds and their applications to algebraic geometry},1986.

\bibitem{Ts} H. Tsuji, {\em A characterization of ball quotients with smooth boundary}, Duke Math. J. 57(1988), 537-553.

\bibitem{Wu} D. Wu, {\em Higher canonical asymptotics of K\"ahler-Einstein metrics on quasi-projective manifolds}, Commu. Analysis and Geometry, 14(2006), 795-845.

\bibitem{Yau} S.T. Yau, {\em On the Ricci curvature of a compact K\"ahler manifold and the complex Monge-Amp\`ere equation. I.},  Comm. Pure Appl. Math. 31 (1978), no. 3, 339-411.

\bibitem{YZ} H. Yin and K. Zheng, {\em Expansion formula for complex Monge-Amp\`ere equation along cone singularities}, arXiv:1609.03111.






\end{thebibliography}
\end{document}